\documentclass[11pt]{article}

\usepackage{amsmath, amsthm}
\usepackage{amsmath, amsfonts}
\usepackage{amsmath, amssymb}
\usepackage{amsmath}
\usepackage[all]{xy}

\newtheorem{definition-lemma}[theorem]{Definition/Lemma}
\newtheorem{definition-explanation}[theorem]{Definition/Explanation}
\newtheorem{explanation-definition}[theorem]{Explanation/Definition}

\newtheorem{lemma-definition}[theorem]{Lemma/Definition}

\numberwithin{equation}{subsection}

\newtheorem{stheorem}{Theorem}[section]
\newtheorem{sdefinition}[stheorem]{Definition}
\newtheorem{sdefinition-lemma}[stheorem]{Definition/Lemma}
\newtheorem{sdefinition-explanation}[stheorem]{Definition/Explanation}
\newtheorem{sexplanation-definition}[stheorem]{Explanation/Definition}
\newtheorem{slemma}[stheorem]{Lemma}
\newtheorem{slemma-definition}[stheorem]{Lemma/Definition}

\newtheorem{scorollary}[stheorem]{Corollary}
\newtheorem{sremark}[stheorem]{\it Remark}

\input{epsfig.sty}
\newcommand{\Ad}{\mbox{\it Ad}}

 \newcommand{\Azscriptsize}{\mbox{\scriptsize\it A$\!$z}}

\newcommand{\Diag}{\mbox{\it Diag}\,}

\newcommand{\End}{\mbox{\it End}\,}
 \newcommand{\smallEnd}{\mbox{\small\it End}\,}
\newcommand{\Endsheaf}{\mbox{\it ${\cal E}\!$nd}\,}

\newcommand{\GL}{\mbox{\it GL}}

\newcommand{\Id}{\mbox{\it Id}}
 
\newcommand{\IGL}{\mbox{\it IGL}}

\newcommand{\Ker}{\mbox{\it Ker}\,}

\newcommand{\Rep}{\mbox{\it Rep}\,}

\newcommand{\Span}{\mbox{\it Span}\,}
\newcommand{\Spec}{\mbox{\it Spec}\,}
 \newcommand{\boldSpec}{\mbox{\it\bf Spec}\,}
 
 \newcommand{\smallSpec}{\mbox{\small\it Spec}\,}

\newcommand{\Stab}{\mbox{\it Stab}\,}

\newcommand{\SU}{\mbox{\it SU}}
\newcommand{\Supp}{\mbox{\it Supp}\,}
 \newcommand{\scriptsizeSupp}{\mbox{\scriptsize\it Supp}\,}
\newcommand{\Sym}{\mbox{\it Sym}}

\newcommand{\Wilson}{\mbox{\it\tiny Wilson}}

\newcommand{\dimm}{\mbox{\it dim}\,}

\newcommand{\orbil}{\mbox{\it orbi.l\,}}

\newcommand{\pr}{\mbox{\it pr}}
\newcommand{\pt}{\mbox{\it pt}}

\newcommand{\longlongmapsto}
 {\raisebox{-.22ex}{$\vdash$}\hspace{-.8ex}
   \raisebox{.51ex}{\rule{6em}{.11ex}}\hspace{-.8ex}\rightarrow}

\begin{document}

\enlargethispage{23cm}

\begin{titlepage}

$ $

\vspace{-1cm}
% \vspace{-1.5cm} % Re: -1.5cm for PC; -2.5cm for UT-Math-system

\noindent\hspace{-1cm}
\parbox{6cm}{\small December 2008}\
   \hspace{8cm}\
   \parbox[t]{5cm}{yymm.nnnn [math.AG]\\ D(3): D0, ADE.}

\vspace{2cm}

%title
\centerline{\large\bf
 Azumaya structure on D-branes}
\vspace{1ex}
\centerline{\large\bf
 and resolution of ADE orbifold singularities revisited:}
\vspace{1ex}
\centerline{\large\bf
 Douglas-Moore vs.\ Polchinski-Grothendieck}
% end-title

\bigskip

\vspace{3em}

%authors-'n-addresses
\centerline{\large
  Chien-Hao Liu
  \hspace{1ex} and \hspace{1ex}
  Shing-Tung Yau
}

\vspace{6em}

%abstract%
\begin{quotation}
\centerline{\bf Abstract}

\vspace{0.3cm}

\baselineskip 12pt  %13pt for [12pt] style
{\small
 In this continuation of [L-Y1] and [L-L-S-Y],
  we explain
   how the Azumaya structure on D-branes
    together with a netted categorical quotient construction
  produces the same resolution of ADE orbifold singularities
  as that arises as the vacuum manifold/variety of
   the supersymmetric quantum field theory
    on the D-brane probe world-volume,
   given by Douglas and Moore [D-M] under the string-theory contents
    and
   constructed earlier through hyper-K\"{a}hler quotients
    by Kronheimer and Nakajima.
 This is consistent with the moral behind this project
  that Azumaya-type structure on D-branes themselves
  -- stated as the Polchinski-Grothendieck Ansatz in [L-Y1] --
  gives a mathematical reason for many
  originally-open-string-induced properties of D-branes.
} % endsmall
\end{quotation}

%\bigskip
\vspace{10em}

\baselineskip 12pt
{\footnotesize
} %end-footnotesize

\bigskip

\noindent {\small MSC number 2000:
 14E15, 81T30; 14A22, 14D21, 81T75.
} % end-small

\bigskip

\baselineskip 10pt
% Re: 11pt for [11pt] style; 12pt for [12pt] style
{\scriptsize
\noindent{\bf Acknowledgements.}
 We thank
  Andrew Strominger and Cumrun Vafa,
   who continue to influence our understanding of topics in string theory.
 C.-H.L.\ thanks in addition
  G\'{e}rard Laumon
   for an illumination on stacks;
  Si Li and Ruifang Song
   for participating in the biweekly
    Saturday D-brane working seminar, spring 2008;
  Frederik Denef, Duiliu-Emanuel Diaconescu, Liang Kong,
  Eric Sharpe, Andrew Tolland
   for communications;
  David Morrison for illuminations of his talk;
  Kefeng Liu, Xiaojun Liu,
  Wenxuan Lu, Li-Sheng Tseng, Peng Zhang,
  Jian Zhou
   for conversations;
  organizers, participants, and staff of the workshop
   ``Algebraic Geometry and Physics, Hangzhou, 2008"
   at Zhejiang University, China;
  Ling-Miao Chou for moral support.
 The project is supported by NSF grants DMS-9803347 and DMS-0074329.
} %endscriptsize

\end{titlepage}

\newpage
\begin{titlepage}

$ $

\vspace{12em} % \vspace{4em}

\centerline{\it
 In memory of Professor Sidney Coleman, 1937 - 2007.$^{\dagger}$}

\vspace{32em}

\baselineskip 10pt
{\scriptsize
\noindent$^{\dagger}${\bf Special tribute from C.-H.L.}
 I was very lucky to to get the chance
  to attend the one-year quantum field theory course,
  Physics 253a and Physics 253b,
  given by Prof.\ Coleman, fall 2000 - spring 2001 at Harvard,
  after having studied his book, ``Aspects of Symmetry",
   and heard of his legendary QFT lectures/course
   from Jacques Distler and Brian Greene on various occasions.
 The lecture notes of the whole year (705 pp.)\ are very impressive
  and
 the style of giving stories/anecdotes from his reflections/witness
  of the development of quantum field theory and particle physics
  up to 1970s along with the lectures
  is very unique.
 I still remember the discussion with him,
  when he lectured on the asymptotic freedom of
  nonabelian gauge field theories near the end of that academic year,
  on Wilson's theory-space and renormalization group flow.
 These latter two notions, together with rigidity of supersymmetry,
   are very fundamental to understanding stringy dualities
  and should be regarded as providing a master moduli space
   that unifies the various moduli spaces in mathematics.
 They still await mathematicians to unravel.
} %endscriptsize

\end{titlepage}

%paper
\newpage
$ $

\vspace{-4em}  % Re: -4cm for PC; -6cm for UT-Math-system

%short heading
\centerline{\sc
 D-Brane Probe Resolution of ADE Orbifold Singularities}

\vspace{2em}

\baselineskip 14pt  %Re: 14pt for [11pt] style
                    %Re: 15pt for [12pt] style.

\begin{flushleft}
{\Large\bf 0. Introduction and outline.}
\end{flushleft}
In this continuation of [L-Y1] and [L-L-S-Y],
 we explain
  how the Azumaya structure on D-branes
   together with a netted categorical quotient construction
 produces the same resolution of ADE orbifold sigularities
 as that arises as the vacuum manifold/variety of
  the open-string-induced supersymmetric quantum field theory
   on the D-brane probe world-volume,
  given by Douglas and Moore [D-M]
   under the string-theory contents  and
  constructed earlier through hyper-K\"{a}hler quotients
   by Kronheimer and Nakajima ([Kr], [K-N], [Na]).

\bigskip

\begin{flushleft}
{\bf Azumaya-type structure on D-branes.}
\end{flushleft}
([Po2: vol.~I: Sec.~8.7; vol.~II: Chap.~13],
 [L-Y1: Sec.~1, Sec.~2], and [L-L-S-Y: Sec.~1].)
Originally, a {\it  D-brane} in string theory
 is by definition an embedding of a manifold/variety/cycle
 in the space-time that serves as a boundary condition
 for open-strings moving in the space-time.
This operational definition allows string-theorists to deduce
 properties of, quantum field theories on, and dynamics of D-branes.
When a few D-branes are stacked together,
 open-string dictates
 that a certain noncommutative geometry emerges.
This noncommutativity feature
  when viewed from Grothendieck's construction/notion of
   a ``geometry" and their morphisms (cf.\ [Ha]),
 taking into account the fact that
  unital associative rings are more natural to do geometries
   from local to global via gluing,
 says that:

\bigskip

\noindent
{\bf Polchinski-Grothendieck Ansatz
     [D-brane: noncommutativity].}\footnote{{\it For
       string-theorists concerning the naming of this ansatz}$\,$:
      This ansatz is nothing more than a rephrasing of
       a few related paragraphs in Polchinski's lecture notes
       or textbooks ([Po2: vol.~I, Sec.~8.7])
       in terms of Grothendieck's aspect of
       local geometries versus coordinate rings,
       cf.\ [Ha] and [L-Y1: Sec.~2.2].
      Together with the fact that
        Polchinski was an early pioneer on D-branes
        and has made a special contribution to the understanding
        of the role of D-branes in string theories [Po1],
       we decided to use this name from the very start of the project.
      A thought came to us after [L-Y1] as whether this is the best name
       to reflect the content of this ansatz.
      Note that this ansatz has nothing to do with supersymmetry.
      It reflects actually only the nature of the enhancement
       of open-string-induced massless spectrum on stacked D-branes,
       particularly the scalar fields thereupon that describe
       deformations of D-branes in space-time.
      This stringy feature on D-branes is in turn a reflection
       of the fact that
       the tension of a fundamental string is a constant in nature
        and, hence, in particular
       the mass of a fundamental open string and
         of the spectrum it creates on D-branes
        are proportional to its length.
      Thus, after seeking the advice of a string theorist
       - who himself is also a figure on the study of branes before 1995
          and who agreed with our naming -,
       we fixed on our original name:
       one string-theorist vs.\ one mathematician,
        with each giving a revolutionary change of the landscape
        of their respective field.
      This naming also takes into account that, since 1995,
       the realization of D-``branes" has broadened considerably and
       this ansatz addresses only the region of the related
       Wilson's theory-space from string theory
       where a D-brane remains a brane.
      See [L-Y1] for more comments.
      Special thanks to Lubos Motl
       for comments that came to my attention accidentally,
       which propelled me to re-think about this naming
       while preparing this manuscript and, hence,
       led to this footnote.
      \hspace{3em}---$\,$ Sincerely, C.-H.L. 2008.12.21.
      %%%%%%%%%%%%%%%
      % PS.\ String-theorists should feel free to email me
      %      any of your comments, particularly the negative ones.\\
      %   $\rule{1.5em}{0em}$ My email can be found through axXiv.
      %%%%%%%%%%%%%%%%
      % Having explained our naming behind the scene,
      % it should however be pointed out that
      %  D-brane as studied in this project reflects only a facet of it
      %   and is most close to what ``D-brane" means in [Po2].
      % The notion of D-``branes" has broadened considerably
      %  (but is still unified physically under the very notion of
      %  ``boundary condition" for open strings) since 1995
      %  due to the numerous collective effort of
      %   both string-theorists and mathematicians,
      % cf.\ the short list in [L-Y1: References].
      %%%%%%%%%%%%%%%%%
      }
 {\rm ([L-Y1: Sec.~2.2].)}
 {\it
 A D-brane (or D-brane world-volume) $X$ carries an Azumaya-type
  noncommutative structure locally associated to a function ring
  of the form $M_r(R)$ for some $r\in {\Bbb Z}_{\ge 1}$ and
  (possibly noncommutative) ring $R$.
 Here, $M_r(R)$ is the $r\times r$ matrix-ring over $R$.
} % end-ansatz

\bigskip

\noindent
Based on this ansatz, a D-brane in a space(-time) $Y$
 in the sense of Polchinski and others in the decade 1986-1995
 can be rephrased prototypically as
 a morphism $\varphi: (X^{A\!z},{\cal E})\rightarrow Y$
  from an Azumaya manifold/scheme $X^{A\!z}$
   with a fundamental module ${\cal E}$
  to $Y$,
 cf.\ [L-Y1: Sec.~1.1, Sec.~1.2, Definition~2.2.3] and
  [L-L-S-Y: Sec.~1, Sec.~2.1, Definition~2.1.2].

Deformations of such morphisms reproduce
 the pattern of gauge-Higgsing/gauge-un-Higgsing behaviors of D-branes
 ([L-Y1: Remark~3.2.4, Sec.~4],
  [L-L-S-Y: Remark 2.1.7, {\sc Figure} 2-1-1], and [Liu1])  and
the moduli space ${\mathfrak M}^{\mbox{\scriptsize D-brane}}(Y)$
 of such morphisms reveals a feature
 as a master moduli space that simultaneously incorporates
 several standard moduli spaces in commutative algebraic geometry
 ([L-Y1: Sec.~4], [L-L-S-Y: Sec.~4], and [Liu2]).
This latter feature is consistent
 with the very robust and versatile nature/role of D-branes
 to serve as a medium/broker/catalyst for various stringy dualities
 since Polchinski's work [Po1] in year 1995.

In this sequel to [L-Y1] and [L-L-S-Y],
we proceed to justify the third known feature of D-branes, namely
  the ability to resolve singularities of the target space(-time),
 along the line of the Polchinski-Grothendieck Ansatz.

\bigskip

\begin{flushleft}
{\bf D-brane probe resolution of ADE orbifold singularities
     \`{a} la Douglas and Moore.}
\end{flushleft}
In the work [D-M] of Douglas and Moore\footnote{Readers
                   are referred to the references of
                   [D-M], [J-M], [D-G-M], [G-L-R] for more literatures
                   on related stringy works and notions that influence
                   the development.
                  Our work proceeds specifically with [D-M] in mind
                   as the inner/compactified part of their setting
                   is a D0-brane moving on an orbifold.
                  The latter is the case we will study.},
the $d=6$, $N=1$ supersymmetric effective field theory (SQFT)
 on the D$5$-branes world-volume $X$
  that is embedded in the product space-time
  ${\Bbb M}^{\,5+1}\times ({\Bbb C}^2/{\Bbb Z}_r)$
  as ${\Bbb M}^{\,5+1}\times\{{\mathbf 0}\}$
 is studied in detail.
Here ${\Bbb M}^{\,5+1}$ is the $6$-dimensional Minkowski
space-time
 and ${\Bbb C}^2/{\Bbb Z}_r$ is the quotient of ${\Bbb C}^2$
  by a ${\Bbb Z}_r$-action
   via a group-embedding ${\Bbb Z}_r\hookrightarrow \SU(2)$.
The {\it massless multiplets} of this $d=6$, $N=1$ SQFT on the brane
 consists of$\,$:
  \begin{itemize}
   \item[$\cdot$]
    [{\it closed-superstring-induced sector}]$\;$
    those from the {\it Kaluza-Klein compactification/reduction}
     of a chosen $d=10$ superstring theory
     on the internal ${\Bbb C}^2/{\Bbb Z}_r$,
      identifying the embedded D$5$-brane world-volume
       with the $d=6$ effective space-time,

   \item[$\cdot$]
    [{\it open-superstring-induced sector}]$\;$
    those {\it from open strings} with one or both end-points
     attached to the D-brane.
  \end{itemize}
The {\it twisted sectors from orbifolding} via the ${\Bbb Z}_r$-action
 are taken into account.
In particular, the {\it scalar fields that describe the deformations
 of this D-brane}, sitting at the orbifold singularity,
 are contained in the hypermultiplets of the theory.
The combinatorial type of this field content
 on the D-brane world-volume is coded in a {\it quiver diagram}.

The {\it Lagrangian} for these supermultiplets
  that governs the low-energy dynamics of the D-brane
 includes a Born-Infeld term, a term for hypermultiplets,
  a Chern-Simons term, and the SUSY completion of these terms.
The {\it space of vacua} for this D-brane
 - computed from
   \begin{itemize}
    \item[$\cdot$]
    the {\it potential} for the hypermultiplets,

    \item[$\cdot$]
    a path-integral manipulation
     to {\it integrate out the auxiliary $D$-fields}
     of the vector multiplets in the Fayet-Iliopoulos terms,

    \item[$\cdot$]
    a {\it condensation} of the scalar fields from the NS-NS sector
     in the Kaluza-Klein compactification
   \end{itemize} -
 gives a resolution of the singular space ${\Bbb C}^2/{\Bbb Z}_r$
  for a generic choice of the vacuum expectation value (vev)
  for the condensation.
In other words,
 {\it the geometry of the inner space ${\Bbb C}^2/{\Bbb Z}_r$
  at an ultrashort distance in string theory
 as seen by this D$5$-brane},
    transverse to ${\Bbb C}^2/{\Bbb Z}_r$
     at the orbifold singularity,
 {\it can be different from the original ${\Bbb C}^2/{\Bbb Z}_r$
      to begin with}.
In particular, this geometry seen by the D-brane can be smooth.
This gives rise to the phenomenon of
 {\it D-brane probe resolution of singularities} of a space.

A generalization of [D-M] to all ADE orbifolds ${\Bbb C}^2/\Gamma$
 is given later in the work [J-M] of Johnson and Myers
 from a slightly different D-brane aspect but with similar conclusion
 on the D-brane probe resolution of ADE orbifold singularities.
Further studies along this line, e.g.\ [D-G-M] and [G-L-R],
 were made for other spaces with orbifold singularities.

\bigskip

\begin{flushleft}
{\bf Geometric-invariant-theory (GIT) quotient
     vs.\ netted categorical quotient.}
\end{flushleft}
The construction of the vacuum manifold/variety in Douglas-Moore
[D-M]
 matches with the hyper-K\"{a}hler quotient construction of
 Kronheimer and Nakajima ([Kr], [K-N]; see also [Na]).
In the algebrao-geometric setting, this is closely related to
 Munmford's geometric-invariant-theory (GIT) quotient construction
 ([M-F-K]).

In general, for an algebraic group $G$ acting on a scheme $Z$
  with a categorical quotient $Z/\!\sim$,
 the stabilizers of the action can be too big to allow the application
  of the GIT construction to produce other quotient spaces from $Z$.
One can try to enhance $Z$ to another $G$-scheme $\widetilde{Z}$
  with $G$-equivariant morphism $\pi:\widetilde{Z}\rightarrow Z$
 to reduce the stabilizers and perform the GIT construction
 on $\widetilde{Z}$.
When it works,
different choices of $\widetilde{Z}$ with a $G$-linearized line bundle
  $\widetilde{L}^{\chi}$ on $\widetilde{Z}$
 give rise then to a net of $G$-invariant
  $\pi(\widetilde{Z}^{ss}(\widetilde{L}^{\chi}))\subset Z$,
 whose categorical quotients
  form now a net of quotient spaces with a natural morphism to $Z/\!\sim$.
Here, $\widetilde{Z}^{ss}(\widetilde{L}^{\chi})$ is the semistable locus
 of $\widetilde{Z}$ with respect to the line bundle $\widetilde{L}$
  on $\widetilde{Z}$ with the $G$-linearization $\chi$.
We will call this procedure
 a {\it netted categorical quotient construction}
 on the $G$-scheme $Z$.
In good cases, one can apply this to produce
 a net of birational models for the (usually bad/singular)
 scheme $Z/\!\sim$.

Readers are referred to [M-F-K] and [Ne]
 for detailed discussions
 on moduli problems, orbit spaces, and the GIT construction.

\bigskip

\begin{flushleft}
{\bf D-brane probe and birational geometry.}
\end{flushleft}
Associated to a combinatorial type
  (e.g.\ dimension and the number of susy's)
 of quantum field theories is
 the {\it Wilson's theory-space} ${\cal S}_{\Wilson}$
 that parameterizes all the quantum field theories
 of the given combinatorial type.
In simple cases, ${\cal S}_{\Wilson}$ is locally parameterized by
 the tuple of coupling constants in the Lagrangian for
 the quantum field theories.
For convenience,
 one may also add in the space on which the condensation of fields
  takes value.
The vacuum manifold/variety of a quantum field theory
 depends on the coupling constants and the condensation values
 and, hence, on where we are on ${\cal S}_{\Wilson}$.
Moving around in ${\cal S}_{\Wilson}$ may give rise to
 a web/net of vacuum manifolds/varieties of different topologies.
{\it Walls} can form on ${\cal S}_{\Wilson}$ that locally separates
 quantum field theories of the same combinatorial type
 but of different details,
 e.g.\ with different topologies of the vacuum manifold/variety.
This gives a {\it phase structure}\footnote{It
                    should be noted
                   that ${\cal S}_{\Wilson}$ can be
                    non-smooth and with several irreducible components.
                   In such a case, besides the transitions
                     due to crossing the walls
                     inside each irreducible component,
                    there are also transitions due to moving
                     from one irreducible component
                     of ${\cal S}_{\Wilson}$ to another.}
 on ${\cal S}_{\Wilson}$.

Applying this picture to the SQFT on a D-brane probe as in [D-M]
 can give rise to a web/net of birational models of the singular space
 in question by taking the vacuum manifold associated to
 a [theory] $\in {\cal S}_{\Wilson}^{\,\mbox{\tiny D-brane probe}}$.
See, e.g., [G-L-R] for an example of flips-and-flops
 of vacuum manifolds of D-brane probes.

As we are studying D-branes
 along the line of the Polchinski-Grothendieck Ansatz,
the D-brane moduli space arises from the moduli stack
 ${\mathfrak M}^{\mbox{\scriptsize D-brane}}(Y)$
 of morphisms from Azumaya spaces with a fundamental module
 to a string target-space $Y$.
An important guiding/natural question is thus:
 \begin{itemize}
  \item[{\bf Q.}] {\bf [D-brane probe resolution]}$\;$\\
  {\it Can one extract birational models,
       in particular resolutions, of $Y$
       from ${\mathfrak M}^{\mbox{\scriptsize D-brane}}(Y)$?}
 \end{itemize}
In this work, we will give an affirmative explicit answer
 to this question for $Y$ being an ADE orbifold $[{\Bbb C}^2/\Gamma]$.

\bigskip

\noindent
{\it Remark 0.1} [{\it naturality}]$\;$
 It should be noted that
  the production of a web/net of birational models
   of an open-string target space/variety
   from a sub-moduli space of the D-brane probe moduli space
         in the sense of {\rm [L-Y1]} and {\rm [L-L-S-Y]}
   by the netted categorical quotient construction
  is a very general outcome of the Azumaya nature on D-branes.
 In particular, it applies not just
  to the case of ADE orbifold singularities.
 One should think of this setting as
  an extension of the application of the {\it functor of points}
       associated to a given target-space
      to a special noncommutative, namely the Azumaya type,
       source geometry.
 Specifically for D$0$-branes,
  allowing the rank of the fundamental/Chan-Paton module
   on D$0$-brane probes to increase,
  this then generalizes the notion of {\it arcs} and {\it jet schemes}
   of a target-space in commutative algebraic geometry.
 From this last point of view,
  the close tie of D-branes,
    along the line of the Polchiski-Grothendieck Ansatz,
   with birational geometry and resolution of singularities
  is very natural and anticipated.\footnote{String-theorists
                             should {\it not} misunderstand us
                             as saying that
                             what string theorists did based on
                              supersymmetric quantum field theory setting
                              is just trivial.
                             Completely the opposite:
                              what string theorists have taught
                               mathematicians
                               are highly non-trivial!
                             Rather, what we mean to say is that
                              the Polchinski-Grothendieck Ansatz
                              together with mathematical naturality
                              alone are enough to ``foresee"
                              what could/should happen.
                             The fact that the mathematical outcome
                                based solely on this
                               and the string-theoretical outcome
                               can be arranged to agree
                              tells us that
                              this ansatz does characterize
                              a very fundamental nature of D-branes.
                             Indeed,
                              from the Grothendieck's point of view
                              of the contravariant equivalence of
                              the category of
                               functions rings (commutative or not)
                              and the category of local geometries,
                             there is no other way/option things can be
                              than the emergence of
                               an Azumaya-type noncommutative structure
                               on stacked D-branes themselves
                              if D-branes have to behave
                               as open strings would dictate.
                             From Grothendieck's viewpoint,
                              it is the noncommutative structure
                               on a D-brane itself
                              that will in turn enable it to probe
                               the noncommutativity, if any,
                               of the string target space(-time);
                             cf.\ [L-Y1: Sec.~2.2] and
                                  [L-L-S-Y: {\sc Figure}~1-2].}
Cf.\ {\sc Figure}~0-1.
\begin{figure}[htbp]
 \epsfig{figure=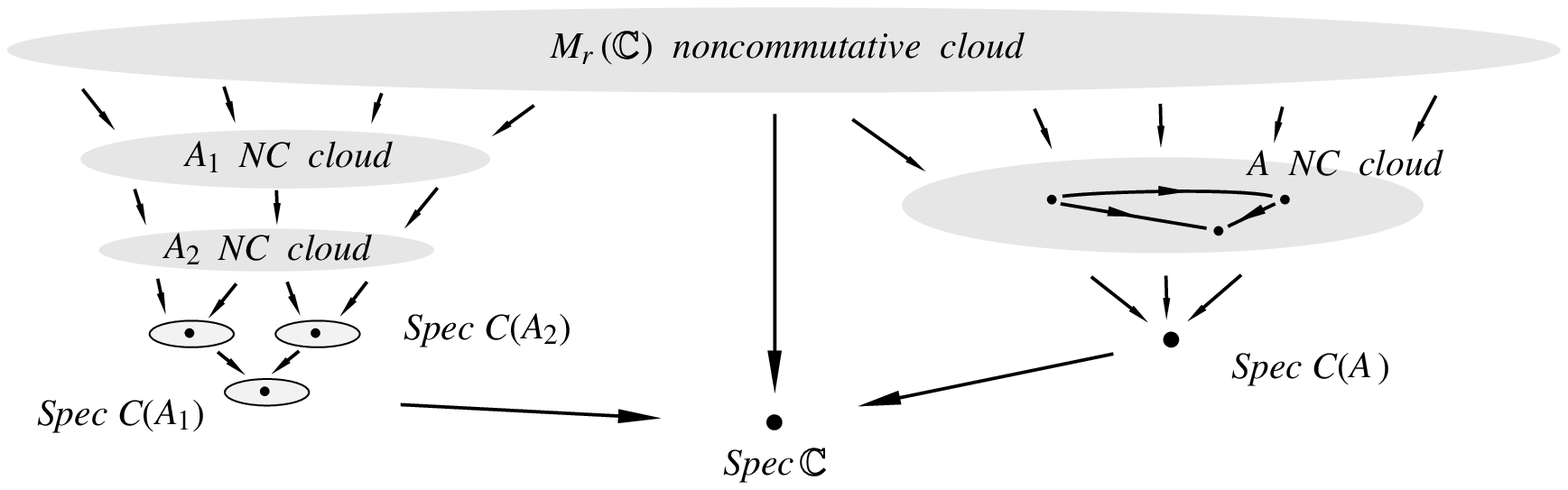,width=16cm}

 \vspace{3em}

 \centerline{\parbox{13cm}{\small\baselineskip 13pt
  {\sc Figure}~0-1.
  An Azumaya scheme contains a very rich amount of geometry,
   revealed via its surrogates;
   cf.\ [L-L-S-Y: {\sc Figure}~1-3].
  Indicated here is the geometry of an Azumaya point
   $\pt^{\Azscriptsize} := (\smallSpec{\Bbb C}, M_r({\Bbb C}))$.
  Here, $A_i$ are ${\Bbb C}$-subalgebras of $M_r({\Bbb C})$
    and $C(A_i)$ is the center of $A_i$ with
   $$
     \begin{array}{cccccccl}
      M_r({\Bbb C}) & \supset  & A_1  & \supset  &  A_2
                    &\supset   & \cdots \\
       \cup  && \cup && \cup \\
     {\Bbb C}\cdot {\mathbf 1} & \subset  & C(A_1)
                    & \subset  & C(A_2)   & \subset  & \cdots & .
     \end{array}
   $$
  According to the Polchinski-Grothendieck Ansatz,
   a D$0$-brane can be modelled prototypically
   by an Azumaya point with a fundamental module of type $r$,
    $(\smallSpec{\Bbb C},\smallEnd({\Bbb C}^r),{\Bbb C}^r)$.
  When the target space $Y$ is commutative,
   the surrogates involved are commutative ${\Bbb C}$-sub-algebras
    of the matrix algebra $M_r({\Bbb C})=\End({\Bbb C}^r)$.
  This part already contains an equal amount of
   information/richness/complexity
   as the moduli space of $0$-dimensional coherent sheaves
   of length $r$.
  When the target space is noncommutative,
   more surrogates to the Azumaya point will be involved.
  Allowing $r$ to go to $\infty$ enables Azumaya points to probe
   ``infinitesimally nearby points" to points on a scheme
   to arbitrary level/order/depth.
  In (commutative) algebraic geometry,
   a resolution of a scheme $Y$ comes from a blow-up.
  In other words,
   a resolution of a singularity $p$ of $Y$ is achieved
   by adding an appropriate family of
   infinitesimally nearby points to $p$.
  Since D-branes with an Azumaya-type structure
   are able to ``see" these infinitesimally nearby points
    via morphisms therefrom to $Y$,
   they can be used to resolve singularities of $Y$.
  Thus, from the viewpoint of Polchinski-Grothendieck Ansatz,
   the Azumaya-type structure on D-branes is
   why D-branes have the power to ``see" a singularity
   of a scheme not just as a point,
   but rather as a partial or complete resolution of it.
  Such effect should be regarded as a generaliztion of
   the standard technique in algebraic geometry
   of probing a singularity of a scheme by arcs
   of the form $\Spec({\Bbb C}[\varepsilon]/(\varepsilon^r))$,
   which leads to the notion of jet-schemes
   in the study of singularity and birational geometry.}}
\end{figure}

\bigskip

\noindent
{\bf Convention.}
 Standard notations, terminology, operations, facts in
  (1) (super)string theory;
  (2) (commutative) algebraic geometry;
  (3) (commutative) stacks;
  (4) descent theory
  can be found respectively in
  (1) [Po2];
  (2) [Ha];
  (3) [L-MB];
  (4) [SGA1] and [Vi].
 \begin{itemize}
  \item[$\cdot$]
   All schemes, algebraic stacks are Noetherian over ${\Bbb C}$.
   All coherent sheaves of modules on an algebriac stack are Cartesian.

  \item[$\cdot$]
   To make the discussion more down to earth,
   the $Y$ in a {\it presentation} $P:Y\rightarrow {\cal Y}$
    of an algebraic stack ${\cal Y}/S$ over a base scheme $S$
    will be chosen by definition to be a scheme$/S$,
    instead of an algebraic space$/S$;
   we will call $Y$ also directly as an {\it atlas} of ${\cal Y}$.
   Similarly, for open charts, $\ldots\,$, etc..

  \item[$\cdot$]
   For linear spaces, e.g.\ ${\Bbb C}^2$, $\End({\Bbb C}^r)$, $\ldots\,$,
    we occasionally identify them with their canonically associated
    affine variety for simplicity of notations.
   Similarly, for their algebraic subsets.
 \end{itemize}

\bigskip

\begin{flushleft}
{\bf Outline.}
\end{flushleft}
{\small
\baselineskip 11pt  %13pt
\begin{itemize}
 \item[0.]
  Introduction.
  \vspace{-.6ex}
  \begin{itemize}
   \item[$\cdot$]
    Azumaya-type structure on D-branes.

   \item[$\cdot$]
    D-brane probe resolution of ADE orbifold singularities
    \`{a} la Douglas and Moore.

   \item[$\cdot$]
    Geometric-invariant-theory (GIT) quotient vs.\
    netted categorial quotient.

   \item[$\cdot$]
    D-brane probe and birational geometry.

   \end{itemize}

 \item[1.]
  D-branes on a (commutative) algebraic stack.
  \vspace{-.6ex}
  \begin{itemize}
   \item[$\cdot$]
    Coherent sheaves on algebraic stacks and flatness.

   \item[$\cdot$]
    D-branes on an algebraic stack
    \`{a} la Polchinski-Grothendieck Ansatz.

   \item[$\cdot$]
    The case of an orbifold target.
  \end{itemize}

 \item[2.]
  D-brane probe resolution of ADE orbifold singularities revisited
   \`{a} la Polchinski-Grothendieck Ansatz.
  \vspace{-.6ex}
  \begin{itemize}
   \item[$\cdot$]
    The moduli stack ${\mathfrak M}^{D0}_1([{\Bbb A}^2/\Gamma])$
    of D0-branes on $[{\Bbb A}^2/\Gamma]$.

   \item[$\cdot$]
    A digression: the equivalent Azumaya-'n-morphism setting.

   \item[$\cdot$]
    D$0$-branes on ${\Bbb A}^2$.

   \item[$\cdot$]
    D$0$-branes on $[{\Bbb A}^2/\Gamma]$ and
    an atlas of ${\mathfrak M}^{D0}_1([{\Bbb A}^2/\Gamma])$.

   \item[$\cdot$]
    Resolution of the ADE orbifold singularity of ${\Bbb A}^2/\Gamma$
    via ${\mathfrak M}^{D0}_1([{\Bbb A}^2/\Gamma])$.
  \end{itemize}
\end{itemize}
} %endsmall

\newpage
\section{D-branes on a (commutative) algebraic stack.}

A formulation of D-branes on an algebraic stack
 that follows the Polchinski-Grothendieck Ansatz
  and extends [L-Y1] and [L-L-S-Y]
 is given in this section.

\bigskip

\begin{flushleft}
{\bf Coherent sheaves on algebraic stacks and flatness.}
\end{flushleft}
The notion of
 \begin{itemize}
  \item[$\cdot$]
   {\it dimension} of an algebraic stack,

  \item[$\cdot$]
   ({\it Cartesian}) {\it coherent sheaves} on an algebraic stack\\
     and their support, {\it pull-back} and {\it push-forward}
 \end{itemize}
 are defined in [L-MB].
For the notion of the {\it support} $\Supp{\cal F}$ of a coherent
sheaf ${\cal F}$,
 we will use instead
 the one that goes through the ideal sheaf of annihilators
  of the coherent sheaf on an algebraic stack
  that generalizes the notion of
   a scheme-theoretic support of a coherent sheaf on a scheme.
The {\it dimension of a coherent sheaf} ${\cal F}$
 on an algebraic stack ${\cal X}$
 is defined to be the dimension of the support $\Supp{\cal F}$
 of ${\cal F}$ on ${\cal X}$.

\smallskip

\begin{sexplanation-definition}
{\bf [Property P of ${\cal O}$-module on algebraic stack].}\footnote{As
                    this definition is not given explicitly in [L-MB],
                    we particularly like to thank G\'{e}rard Laumon
                    for the very careful and authoritative explanation
                    of this to us.
                   Thanks also to Andrew Tolland for a discussion.}
 {\rm ([Lau].)}
{\rm
 Let P be a property of ${\cal O}$-modules on schemes
  satisfying the following two conditions:
  \begin{itemize}
   \item[(1)] ({\it pull-back})$\,$
   If an ${\cal O}$-module ${\cal F}$ on a scheme $X$
    has the property P,
   then for any smooth morphism of schemes $X^{\prime}\rightarrow X$,
    the pull-back ${\cal F}^{\prime}$ of ${\cal F}$ on $X^{\prime}$
    also has the property P.

   \item[(2)] ({\it descent})$\,$
   An ${\cal O}$-module ${\cal F}$ on a scheme $X$ has the property P
    as soon as
    there exists a smooth and surjective morphism
      of schemes $X^{\prime}\rightarrow X$
     such that
      the pull-back ${\cal F}^{\prime}$ of ${\cal F}$ on $X^{\prime}$
      has the property P.
  \end{itemize}
 Then the {\it property P} makes sense
  {\it for ${\cal O}$-modules on algebraic stacks}$\,$ and
 it is enough to check it on any presentation of the algebraic stack.
}
\end{sexplanation-definition}

\smallskip

\begin{sdefinition}
{\bf [flatness of coherent module on ${\cal X}/{\cal Y}$].}
{\rm
 Given
  a ($1$-)morphism $F:{\cal X}\rightarrow {\cal Y}$
   of algebraic stacks and
  a coherent ${\cal O}_{\cal X}$-module ${\cal M}$ on ${\cal X}$,
 we say that
   ${\cal M}$ is {\it flat over} ${\cal Y}$
  if for every open chart
   $u:U\rightarrow {\cal Y}$
    (in the smooth-\'{e}tale topology/site) of ${\cal Y}$,
  the pull-back $(\widehat{u}\circ v)^{\ast}{\cal M}$
   on every $V$ is flat over $U$,
   where
    $v:V \rightarrow U\times_{u,{\cal Y},F}{\cal X}$
     is an open chart of $U\times_{u,{\cal Y},F}{\cal X}$
    and
    \begin{eqnarray*}
     \xymatrix{
     V\ar[rr]^v\ar[rrd]
      && U\times_{u,{\cal Y},F}{\cal X}
         \ar[rr]^{\widehat{u}}\ar[d]
      && {\cal X}\ar[d]^F \\
     && U\ar[rr]^{u}  && {\cal Y} & \hspace{-4em}.
      }
    \end{eqnarray*}
 (Note that both $U$ and $V$ are schemes by our convention.)
}
\end{sdefinition}

\smallskip

\noindent
This notion plays a fundamental role in the setting of D-branes
 on an algebraic stack studied in the current work.

\bigskip

\begin{flushleft}
{\bf D-branes on an algebraic stack
     \`{a} la Polchinski-Grothendieck Ansatz.}
\end{flushleft}
In [L-Y1: Sec.~1] and [L-L-S-Y: Sec.~2.1] the notion of morphisms
 from an Azumaya scheme with a fundamental module $(X^{Az}, {\cal E})$
 to a projective variety $Y$ is developed from Grothendieck's
 fundamental principle of doing geometries and morphisms
 from local to global via gluing,
together with the natural requirement that
 the composition of morphisms $X_1\rightarrow X_2 \rightarrow X_3$
 should be a morphism $X_1\rightarrow X_3$.
This gives us a prototype definition of a D-brane on $Y$
 along the line of Polchinski-Grothendieck Ansatz.
An algebraic stack and morphisms between them can be constructed
 from local to global with generalized notion of
  coverings and gluing via smooth morphisms.
Thus, one can repeat [L-Y1] and [L-L-S-Y: Sec.~2.1]
 to give a fundamental treatment of
  the notion of
   morphisms from Azumaya schemes with a fundamental module
   to an algebraic stack ${\cal Y}$ and, hence,
  the notion of D-branes on the stack ${\cal Y}$
   following the Polchinski-Grothendieck Ansatz.

In [L-L-S-Y: Sec.~2.2] it is next observed that
 the above fundamental treatment for target a projective variety
 can be recast into an equivalent
 Azumaya-without-Azumaya-'n-morphism-without-morphism setting.
In this latter setting,
 a morphism $\varphi:(X^{\Azscriptsize},{\cal E})\rightarrow Y$,
  where
  $X^{\Azscriptsize}
   =(X,
     {\cal O}^{\Azscriptsize}_X=\Endsheaf_{{\cal O}_X}{\cal E})$
 is completely coded by
 a coherent ${\cal O}_{X\times Y}$-module $\widetilde{\cal E}$
 on $(X\times Y)/X$
 that is flat over $X$ and of relative dimension $0$.
With the notion of flatness of coherent sheaves on an algebraic stack
 over another algebraic stack being set up in the previous theme,
in this work we take the shorter second route
 to define prototypically
 a D-brane on an algebraic stack ${\cal Y}$
 as a morphism
  from an Azumaya scheme with a fundamental module to ${\cal Y}$
 via a direct generalization of
  the Azumaya-without-Azumaya-'n-morphism-without-morphism setting
  of [L-L-S-Y: Sec.~2.2],
 as follows:

\smallskip

\begin{sdefinition}
{\bf [D-brane on algebraic stack \`{a} la Polchinski-Grothendieck Ansatz].}
{\rm
 A {\it D-brane} with the underlying domain (commutative) scheme $X$
  on a target algebraic stack ${\cal Y}$
  {\it \`{a} la Polchinski-Grothendieck Ansatz}
 is defined to be a coherent ${\cal O}_{X\times{\cal Y}}$-module
 $\widetilde{\cal E}$ on $X\times{\cal Y}$
 that
  \begin{itemize}
   \item[(i)]
    is flat over $X$,

   \item[(ii)]
    is of relative dimension $0$ with respect to $X$, and

   \item[(iii)]
    ${\cal E}:=\pr_{1\ast}\widetilde{\cal E}$
     is a locally-free coherent ${\cal O}_X$-module,
     where $\pr_1:X\times {\cal Y}\rightarrow X$
      is the projection map to $X$.
  \end{itemize}
 Two D-branes, represented respectively by
    $\widetilde{\cal E}_1$ on $X_1\times {\cal Y}$ and
    $\widetilde{\cal E}_2$ on $X_2\times {\cal Y}$
    under the above setting,
   are said to be {\it isomorphic}
  if there exit
   a (${\Bbb C}$-)scheme-isomorphism
     $h:X_1\stackrel{\sim}{\rightarrow} X_2$ and
   an ${\cal O}_{X_1\times{\cal Y}}$-module-isomorphism
      $\widetilde{h}:
       h^{\ast}\widetilde{\cal E}_2
         \stackrel{\sim}{\rightarrow} \widetilde{\cal E}_1$.
}\end{sdefinition}

\smallskip

\noindent
Note that for general ${\cal Y}$ in Definition~1.3,
  % Definition [D-brane on algebraic stack
  %              \'{a} la Polchiski-Grothendieck Ansatz]
 Condition (i) and Condition (ii) together do not imply
 Condition (iii).
The latter thus has to be imposed additionally.

\smallskip

\begin{sremark}
{\rm [{\it recovering morphism from Azumaya scheme
           with fundamental module to ${\cal Y}$
           from $\widetilde{\cal E}$}].}
{\rm (Cf.\ [L-L-S-Y: Sec.~2.2 and Figure~2-2-1].)} {\rm
 (1)
 The $\widetilde{\cal E}$ in the above definition reproduces
  \begin{itemize}
   \item[$\cdot$]
    an underlying Azumaya scheme with a fundamental module
     $$
      (X^{A\!z},{\cal E})\;
       :=\;
        (X,\,
         {\cal O}_X^{A\!z}:=\Endsheaf_{{\cal O}_X}{\cal E},\,
         {\cal E})\,,
     $$
     where ${\cal E}= \pr_{1\ast}\widetilde{\cal E}$
     in the Definition,

   \item[$\cdot$]
    a morphism
    $$
     \varphi\; :\;
      (X^{A\!z},{\cal E})\;  \longrightarrow\; {\cal Y}
    $$
    via the restriction of the projection map
     $\pr_2:X\times {\cal Y}\rightarrow {\cal Y}$
     to $\Supp\widetilde{\cal E}$.
  \end{itemize}
 The {\it surrogate} $X_{\varphi}$ of $X^{A\!z}$
  associated to $\varphi$ is given by
  $\boldSpec(
     \pr_{1\ast}({\cal O}_{\scriptsizeSupp\widetilde{\cal E}}) )$.
 This is a scheme that is finite over $X$.
 In terms of $X_{\varphi}$, $\varphi$ is represented by
  a scheme-morphism
  $\widehat{f}_{\varphi}: \widehat{X}_{\varphi}\rightarrow Y$,
   where
    $P:Y\rightarrow {\cal Y}$ is an atlas of ${\cal Y}$ and
    $\widehat{X}_{\varphi}$ is a scheme with a built-in
     smooth surjective morphism
     $\widehat{X}_{\varphi}\rightarrow X_{\varphi}$,
     depending on both $\widetilde{\cal E}$ and the choice of $P$.

 (2)
 The {\it image D-brane with a Chan-Paton module}
  on the stack ${\cal Y}$ is defined/given by the push-forward
  $\varphi_{\ast}{\cal E} := \pr_{2\ast}\widetilde{\cal E}$.
 The support $\Supp(\varphi_{\ast}{\cal E})$ on ${\cal Y}$
  is by definition the underlying ``brane" on ${\cal Y}$. }
\end{sremark}

\smallskip

\begin{slemma}
{\bf [pull-back to atlas].}
{\it
 Continuing Definition~1.3,
          % Definition [D-brane on algebraic stack
          %             \`{a} la Polchinski-Grothendieck Ansatz]
 a coherent ${\cal O}_{X\times{\cal Y}}$-module $\widetilde{\cal E}$
  on $X\times{\cal Y}$ is flat over $X$
 if and only if
  there exists an atlas $P:Y\rightarrow {\cal Y}$
   such that
    $(\Id_X\times P)^{\ast}\widetilde{\cal E}$ on $X\times Y$
    is flat over $X$.
 The latter holds
  if and only if
   $(\Id_X\times P)^{\ast}\widetilde{\cal E}$ on $X\times Y$
    is flat over $X$
   for every atlas $P:Y\rightarrow {\cal Y}$.}
\end{slemma}

\begin{proof}
 This follows from
  the definition of flat modules on schemes,
  [Ha: Proposition 9.2(d)], and
  the descent of morphisms of coherent sheaves
    under a smooth morphism (cf.\ [SGA1] and [Vi]), and
  Explanation/Definition~1.1.
  % Explanation/Definition
  %             [Property P of ${\cal O}$-module on algebraic stack]

\end{proof}

\smallskip

\begin{flushleft}
{\bf The case of an orbifold target.}
\end{flushleft}
An orbifold ${\cal Y}$ is a special Deligne-Mumford stack
 for which a local chart is of the form $(U,\Gamma_U)$  and
 ${\cal Y}$ is locally modelled on Deligne-Mumford stacks
  of the form $[U/\Gamma_U]$, the quotient stack of $U$ by
  $\Gamma$,
 where $\Gamma_U$ is a finite group that acts on the scheme $U$.
The topological language of Thurston [Th: Chap.~13] on orbifolds
 is directly adaptable to the algebro-geometric language
 of Deligne-Mumford stacks.
In particular:

\smallskip

\begin{sdefinition-lemma}
{\bf [orbifold structure group at point].} {\rm
 Let
  $p\in {\cal Y}$ be a geometric point of the orbifold ${\cal Y}$ and
  $(U,\Gamma_U)$ be an orbifold chart of ${\cal Y}$
   that contains\footnote{For readers unfamiliar with stacks:
         In terms of Deligne-Mumford stack language,
         this means that $\{p\}\times_{\cal Y}U$ is non-empty,
         where the fibered product is taken with respect to
         the built-in morphisms
          $\{p\}\rightarrow {\cal Y}$ and $U\rightarrow {\cal Y}$.}
  $p$.
 Define the {\it orbifold structure group} $\Gamma_p$
  of ${\cal Y}$ {\it at} $p$ to be the stabilizer $\Stab(p)$
  of the $\Gamma$-action on $U$.
 This is well-defined up to group-isomorphisms
  as, up to group-isomorphisms, $\Gamma_p$ is independent of
   the choice of $(U,\Gamma_U)$ that contains $p$.
}\end{sdefinition-lemma}

\smallskip

\begin{sdefinition-lemma}
{\bf [orbifold-length].} {\rm
 Given a coherent $0$-dimensional ${\cal O}_{\cal Y}$-module
   on an orbifold ${\cal Y}$,
 let $|\Supp{\cal F}|=\{p_1,\,\cdots\,,p_k\}$
  be the set of geometric points in the support $\Supp{\cal F}$
  of ${\cal F}$ and
  $l_i$ be the length of ${\cal F}$ at $p_i$.
 Define the {\it orbifold-length} $\orbil({\cal F})$
  of ${\cal F}$ on ${\cal Y}$ to be
  $$
   \orbil({\cal F})\;=\;
    \sum_{i=1}^k\, \frac{l_i}{|\Gamma_i|}\,.
  $$
 This is invariant under flat deformations of ${\cal F}$.
}\end{sdefinition-lemma}

\smallskip

\begin{sremark} {\rm
[{\it fractional orbifold-length}].
 For an orbifold ${\cal Y}$
  with the orbifold structure group $\Gamma_{\bullet}$ being trivial
  on an open dense sub-orbifold of ${\cal Y}$,
 a ${\cal O}_{\cal Y}$-module ${\cal F}$
   with non-integer $\orbil({\cal F})$
  contains a submodule ${\cal F}^{\prime}$ that is supported on
  some orbifold-point(s) with nontrivial structure group
  in such away that ${\cal F}^{\prime}$ cannot be deformed away
  from this(/these) point(s).
 In other words, ${\cal F}$ has an unmovable/trapped
  direct summand supported at point(s) with nontrivial
  orbifold structure group.
}\end{sremark}

\smallskip

\begin{sremark} {\rm
[{\it orbifold Euler characteristic}].
 The notion of orbifold-length is a special case of
 the notion of {\it orbifold Euler characteristic}
 for a coherent ${\cal O}_{\cal Y}$-module on an orbifold ${\cal Y}$.
}\end{sremark}

\bigskip

\section{D-brane probe resolution of ADE orbifold singularities
         revisited \`{a} la Polchinski-Grothendieck Ansatz.}

We now address the main theme of the current work:
extracting a resolution
 for the variety ${\Bbb A}^2/\Gamma$
 from a D-brane probe moduli space
 in the sense of Polchinski-Grothendieck Ansatz.

\bigskip

\begin{flushleft}
{\bf The moduli stack ${\mathfrak M}^{D0}_1([{\Bbb A}^2/\Gamma])$
      of D0-branes on $[{\Bbb A}^2/\Gamma]$.}
\end{flushleft}
Let
 $\Gamma$ be a finite subgroup of $\SU(2)$ acting on
  ${\Bbb A}^2 =\Spec{\Bbb C}[z_1,z_2]$
  via the (anti-)action\footnote{An
        {\it anti-action} of $\Gamma$ is by definition an action
        of $\Gamma^{\circ}$,
        where $\Gamma^{\circ}$ is the group
        that has the same elements as $\Gamma$
        but with $\gamma_1\gamma_2$ in $\Gamma^{\circ}$ defined
        to be $\gamma_2\gamma_1$ in $\Gamma$.
       In this work, there is no chance of ambiguity of
        whether $\Gamma$ acts or anti-acts
        as all the actions or anti-actions involved are induced
        from the $\Gamma$-action on ${\Bbb A}^2$ or its lifting.
       We thus call either directly as an action.}
  $$
   (z_1,z_2) \;
   \stackrel{\scriptsize
             \mbox{\raisebox{3ex}{
     $\gamma=\left(\begin{array}{cc} a&c\\ b&d \end{array} \right)$
              }}} {\longlongmapsto}\;
   (z_1,z_2)
   \left(\begin{array}{cc} a & c \\b & d\end{array} \right)^t \;
   =:\;  (\gamma\odot z_1\,,\, \gamma\odot z_2)
  $$
  on generators of the function ring,
 $r=|\Gamma|$ be the order of $\Gamma$, and
 $[{\Bbb A}^2/\Gamma]$
   be the orbifold associated to the group action.
Note that $(\gamma_2\gamma_1)\odot z_i= \gamma_1\odot (\gamma_2\odot z_i)$,
 $i=1\,, 2$.
By construction,
$[{\Bbb A}^2/\Gamma]$ is a Deligne-Mumford stack
 with an atlas $P:{\Bbb A}^2\rightarrow [{\Bbb A}^2/\Gamma]$.
The fibered product
  \begin{eqnarray*}
  \xymatrix{
   {\Bbb A}^2\times_{P,\,[{\Bbb A}^2/\Gamma],\,P}{\Bbb A}^2
    \ar[d]_{pr_1}\ar[rrr]^{pr_2}
    &&& {\Bbb A}^2\ar[d] \\
   {\Bbb A}^2\ar[rrr]  &&& [{\Bbb A}^2/\Gamma]
  }
  \end{eqnarray*}
  from the Isom-functor construction
 defines a morphism ${\Bbb A}^2\times\Gamma \rightarrow {\Bbb A}^2$
 that recovers the $\Gamma$-action on ${\Bbb A}^2$.
{From} our setting
 Definition~1.3 along the Polchinski-Grothendieck Ansatz,
          % Definition [D-brane on algebraic stack
          %             \`{a} la Polchinski-Grothendieck Ansatz]
 a D$0$-brane on $[{\Bbb A}^2/\Gamma]$ of stacky type $1$ is,
 by definition, given
 by a $0$-dimensional coherent sheaf on $[{\Bbb A}^2/\Gamma]$
  of orbifold-length $1$.
It follows from
 Lemma~1.5 and the above discussion that
     % Lemma [pull-back to atlas]
this coherent ${\cal O}_{[{\Bbb A}^2/\Gamma]}$-module
  on $[{\Bbb A}^2/\Gamma]$
 is identical, via pull-back versus descent, to
  a $0$-dimensional coherent $\Gamma$-${\cal O}_{{\Bbb A}^2}$-module
  $\widetilde{\cal E}$ of length $r$ on ${\Bbb A}^2$.
It follows that:

\smallskip

\begin{slemma}
{\bf [moduli stack of D0-branes on $[{\Bbb A}^2/\Gamma]$].}
{\it
 The moduli stack ${\mathfrak M}^{D0}_1([{\Bbb A}^2/\Gamma])$
   of D$0$-branes of stacky type $1$
    on the orbifold $[{\Bbb A}^2/\Gamma]$
  is an Artin stack,
   given by the stack of
    coherent $\Gamma$-${\cal O}_{{\Bbb A}^2}$-modules
    of length $r$ on ${\Bbb A}^2$,
  where $r=|\Gamma|$.
}
\end{slemma}

\smallskip

\noindent
Cf.\ [L-L-S-Y: Sec.~3.1].

\bigskip

\begin{flushleft}
{\bf A digression:
     the equivalent Azumaya-'n-morphism setting.}\footnote{String-theorists
                          are highly recommended to compare
                          the several very concrete/explicit
                          geometric pictures of D-branes
                           - all following from
                             the Polchinski-Grothendieck Ansatz -
                          in this theme with
                          whatever stringy geometric picture(s)
                          you may have had for Douglas-Moore [D-M].}
\end{flushleft}
While in this work
 the Azumaya-without-Azumaya-'n-morphism-without-morphism setting
  of D-branes along the Polchinski-Grothendieck Ansatz
 is adopted to circumvent the otherwise necessary
  but tedious - albeit more fundamental - presentation for Sec.~1,
it is instructive to see what morphism a geometric point of
 ${\mathfrak M}^{D0}_1([{\Bbb A}^2/\Gamma])$ really corresponds to.
This recovers then the hidden equivalent Azumaya-'n-morphism setting
 of D-branes on the orbifold $[{\Bbb A}^2/\Gamma]$.

Recall Remark~1.4 that generalizes [L-L-S-Y: Sec.~2.2].
     % Remark [recovering morphism from Azumaya scheme
     %         with fundamental module to ${\cal Y}$
     %         to $\widetilde{\cal E}$]
Let
 \begin{eqnarray*}
  \xymatrix{
   [{\Bbb A}^2/\Gamma]^{\bullet}\;
   :=\; \Spec{\Bbb C}\times [{\Bbb A}^2/\Gamma]\;
       (=\; [{\Bbb A}^2/\Gamma])
     \ar[d]^{pr_1}\ar[rrr]^{\hspace{8em}pr_2}
    &&& [{\Bbb A}^2/\Gamma] \\
   \Spec{\Bbb C}
   }
 \end{eqnarray*}
 be the projection maps.
Here we use $[{\Bbb A}^2/\Gamma]^{\bullet}$ to distinguish
 the product (over $\Spec{\Bbb C}$) of
  the domain scheme $\Spec{\Bbb C}$ and
  the target orbifold $[{\Bbb A}^2/\Gamma]$
 from the target orbifold $[{\Bbb A}^2/\Gamma]$
 for conceptual clarity.
Then,
 for an ${\cal O}_{[{\Bbb A}^2/\Gamma]^{\bullet}}$-module
  $\widetilde{\cal E}$ of orbifold-length $1$,
 $\pr_{1\ast}\widetilde{\cal E}\simeq {\Bbb C}^r$,
  where recall that $r=|\Gamma|$.
Thus, a geometric point of ${\mathfrak M}^{D0}_1([{\Bbb A}^2/\Gamma])$
 corresponds to a morphism
 $$
  \varphi\;:\;
   (\pt^{\Azscriptsize},{\Bbb C}^r)\;
   :=\; (\Spec{\Bbb C}, \End({\Bbb C}^r), {\Bbb C}^r)\;
   \longrightarrow\; [{\Bbb A}^2/\Gamma]
 $$
 from the Azumaya ${\Bbb C}$-point of type $r$
  with a fundamental module to $[{\Bbb A}^2/\Gamma]$.
Furthermore, as $\orbil(\widetilde{\cal E})=1$,
 $\Supp\widetilde{\cal E}$ contains only one geometric point $p$
 of $[{\Bbb A}^2/\Gamma]$.
There are only two situations:
 either $\Gamma_p$ is trivial or $\Gamma_p=\Gamma$.
For convenience,
 denote by ${\mathbf 0}$ the geometric point in ${\Bbb A}^2$
  associated to the ideal $(z_1, z_2)$
  for ${\Bbb A}^2=\Spec({\Bbb C}[z_1, z_2])$  and
 let $U = {\Bbb A}^2-\{{\mathbf 0}\}$.
The geometric point in $[{\Bbb A}^2/\Gamma]$ associated to
 ${\mathbf 0}\in {\Bbb A}^2$ will also be denoted by ${\mathbf 0}$.

\medskip

\begin{flushleft}
{{\it Case} $(a)\,$: $\Gamma_p$ is trivial.}
\end{flushleft}
This happens when $p$ lies in the open dense sub-orbifold
 $[U/\Gamma]$ of $[{\Bbb A}^2/\Gamma]$.
In this case$\,$:
 \begin{itemize}
  \item[$\cdot$]
   The length $l_p$ of $\widetilde{\cal E}$ at $p$ is equal to $1$.

  \item[$\cdot$]
   $\Supp\widetilde{\cal E}$ is a $0$-dimension closed sub-orbifold
    $[\{p_1,\,\cdots\,, p_r \}/\Gamma]$ of $[{\Bbb A}^2/\Gamma]$,
    where $\Gamma$ acts on the ($0$-dimensional reduced) scheme
    $\{p_1,\,\cdots\,, p_r \}$ effectively and transitively.

  \item[$\cdot$]
   The surrogate $\pt_{\varphi}$ of $\pt^{\Azscriptsize}$
    associated to $\varphi$ is given by $\{p_1,\,\cdots\,, p_r\}$,
   which is $\Gamma$-isomorphic to
    the disjoint union
    $\Spec(\prod_r{\Bbb C})\;=\;\amalg_r\Spec{\Bbb C}$
    of $r$-many ${\Bbb C}$-points,
    equipped with an effective and transitive $\Gamma$-action
    modelled on the left multiplication of $\Gamma$ on itself.
   Here $\prod_r{\Bbb C}$ is the product ring of ${\Bbb C}$'s and
    is canonically embedded in $\End({\Bbb C}^r)$
    as a ${\Bbb C}$-sub-algebra.
 \end{itemize}
The morphism $\varphi$ is thus represented by
 an embedding of $\Gamma$-schemes:
 $$
  \widehat{f}_{\varphi}\;:\;
   \pt_{\varphi}\;=\; \{p_1,\,\cdots\,, p_r \}\;
    \longrightarrow\; {\Bbb A}^2\,,
 $$
 where ${\Bbb A}^2$ is the built-in atlas of the orbifold
  $[{\Bbb A}^2/\Gamma]$ regarded as a Deligne-Mumford stack.

The fundamental module ${\Bbb C}^r=\pr_{1\ast}\widetilde{\cal E}$
  on $\pt^{\Azscriptsize}$
 is naturally a $\Gamma$-${\cal O}_{pt_{\varphi}}$-module.
In terms of the latter, it is isomorphic to
 ${\cal O}_{pt_{\varphi}} = \oplus_{i=1}^r{\cal O}_{p_i}$.
Note that in the current case,
 $\Gamma$ acts on the fundamental module ${\Bbb C}^r$
 via the regular representation.
The image D$0$-brane with Chan-Paton module on $[{\Bbb A}^2/\Gamma]$
 that corresponds to $\varphi$ is given by
 $\varphi_{\ast}\widetilde{\cal E}
  := \pr_{2\ast}\widetilde{\cal E}$,
 which is $\widetilde{\cal E}$ itself
  after identifying $[{\Bbb A}^2/\Gamma]^{\bullet}$
  with $[{\Bbb A}^2/\Gamma]$ canonically.
It is represented by the $\Gamma$-${\cal O}_{{\Bbb A}^2}$-module
 $\widehat{f}_{\varphi\ast}{\cal O}_{pt_{\varphi}}
  = \oplus_{i=1}^k\,{\cal O}_{\widehat{f}_{\varphi}(p_i)}$
 on the atlas ${\Bbb A}^2$ of $[{\Bbb A}^2/\Gamma]$.

\medskip

\begin{flushleft}
{{\it Case} $(b)\,$: $\Gamma_p=\Gamma$.}
\end{flushleft}
This happens when $p={\mathbf 0}$.
In this case,
 \begin{itemize}
  \item[$\cdot$]
   The length $l_p$ of $\widetilde{\cal E}$ at $p$
    is equal to $r$ ($\,=|\Gamma|\,$).

  \item[$\cdot$]
   $\Supp\widetilde{\cal E}$ is a $0$-dimension closed sub-orbifold
    $[Z/\Gamma]$ of $[{\Bbb A}^2/\Gamma]$.
   Here $\Gamma$ acts on the $0$-dimensional connected scheme $Z$
   and $Z=\Spec A$
    for some local Artin $\Gamma$-${\Bbb C}$-algebra $A$
    of length $\le r$.

  \item[$\cdot$]
   The surrogate $\pt_{\varphi}$ of $\pt^{\Azscriptsize}$
    associated to $\varphi$ is given by the $\Gamma$-scheme $Z$.
 \end{itemize}
The morphism $\varphi$ is thus represented by an embedding
 of $\Gamma$-schemes:
 $$
  \widehat{f}_{\varphi}\;:\;
   \pt_{\varphi}\;=\; Z\;
    \longrightarrow\; {\Bbb A}^2\,,
 $$
 where ${\Bbb A}^2$ is the built-in atlas of $[{\Bbb A}^2/\Gamma]$.

Again, the fundamental module
 ${\Bbb C}^r=\pr_{1\ast}\widetilde{\cal E}$ on $\pt^{\Azscriptsize}$
 is naturally a $\Gamma$-${\cal O}_{pt_{\varphi}}$-module,
 denoted by $\widehat{\cal E}$.
However, it can happen that $\widehat{\cal E}$ is not isomorphic to
 ${\cal O}_{pt_{\varphi}}$.
Note also that $\Gamma$ now acts on the fundamental module ${\Bbb C}^r$,
 but possibly as a direct sum of irreducible representations,
 cf.~{\sc Figure}~2-1: $\varphi_4$.
The image D$0$-brane with Chan-Paton module on $[{\Bbb A}^2/\Gamma]$
 that corresponds to $\varphi$ is given by
 $\varphi_{\ast}\widetilde{\cal E}
  := \pr_{2\ast}\widetilde{\cal E}$,
 which is $\widetilde{\cal E}$ itself
  after identifying $[{\Bbb A}^2/\Gamma]^{\bullet}$
  with $[{\Bbb A}^2/\Gamma]$ canonically.
It is represented by the $\Gamma$-${\cal O}_{{\Bbb A}^2}$-module
 $\widehat{f}_{\varphi\ast}\widehat{{\cal E}}$
 on the atlas ${\Bbb A}^2$ of $[{\Bbb A}^2/\Gamma]$.
Its support $\Supp(\widehat{f}_{\varphi\ast}\widehat{\cal E})$
 on ${\Bbb A}^2$
 is a representation
 of the image D$0$-brane $\Supp(\varphi_{\ast}\widetilde{\cal E})$
  (a sub-orbifold) on $[{\Bbb A}^2/\Gamma]$.

\medskip
\begin{flushleft}
{Cf.\ {\sc Figure} 2-1.}
\end{flushleft}
\begin{figure}[htbp]
 \epsfig{figure=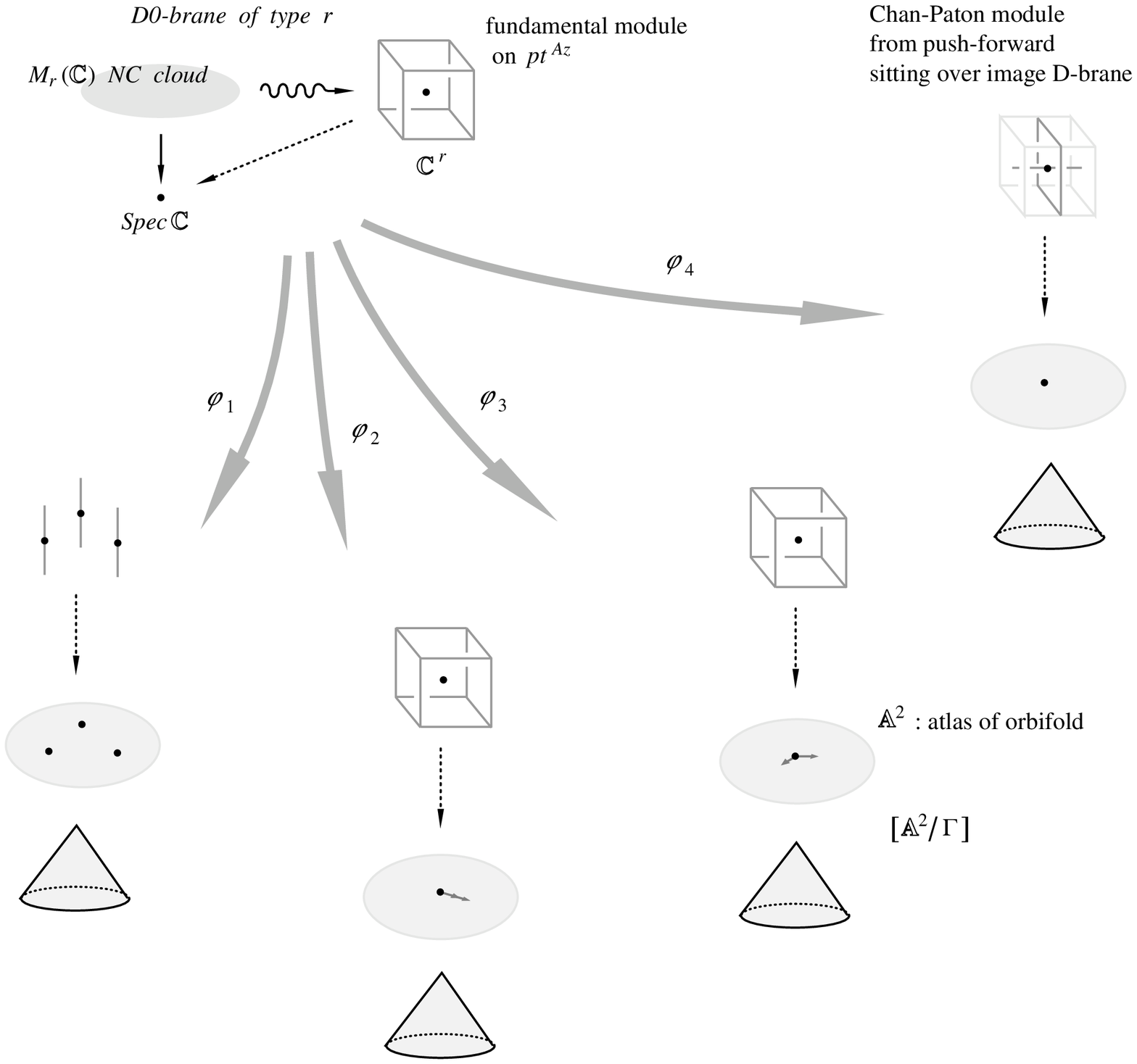,width=16cm}
 \centerline{\parbox{13cm}{\small\baselineskip 13pt
  {\sc Figure} 2-1.
  Examples of morphisms from an Azumaya point with a fundamental module
   $(\smallSpec{\Bbb C}, \smallEnd({\Bbb C}^r),{\Bbb C}^r)$,
    which models an intrinsic D$0$-brane
    according to the Polchinski-Grothendieck Ansatz,
   to the orbifold $[{\Bbb A}^2/\Gamma]$ are shown.
  Morphism $\varphi_1$ is in Case (a)
   while morphisms $\varphi_2$, $\varphi_3$, $\varphi_4$
    are in Case (b).
  The image D$0$-brane under $\varphi_i$ on the orbifold $[{\Bbb A}^2/\Gamma]$
    is represented by a $0$-dimensional $\Gamma$-subscheme of length $\le r$
    on the atlas ${\Bbb A}^2$ of $[{\Bbb A}^2/\Gamma]$.
  }}
\end{figure}

\vspace{16em}

%\bigskip

\begin{flushleft}
{\bf D0-branes on ${\Bbb A}^2$.}
\end{flushleft}
In this theme we recast [L-Y1: Sec.~4.3] as a preparation
 to understanding ${\mathfrak M}^{D0}_1([{\Bbb A}^2/\Gamma])$ further.

A $0$-dimensional coherent ${\cal O}_{{\Bbb A}^2}$-module ${\cal
F}$
   of length $r$ on ${\Bbb A}^2$
  with a specified isomorphism
   $\iota: {\Bbb C}^r\stackrel{\sim}{\rightarrow} H^0({\cal F})$
 can be identified with
  a pair $(m_1,m_2)$ of commuting $r\times r$ matrices$/{\Bbb C}$,
and vice versa, as follows:
\begin{itemize}
 \item[(1)]
 {\it From $({\cal F}, \iota)$ to $(m_1,m_2)$}$\,$:\hspace{1em}
  Given $({\cal F},\iota)$ as said,
  the ${\Bbb C}[z_1,z_2]$-module structure on $ H^0({\cal F})$ and,
   hence, on ${\Bbb C}^r$
   via $\iota:{\Bbb C}^r\stackrel{\sim}{\rightarrow} H^0({\cal F})$,
   gives uniquely a pair $(m_1,m_2)$ as said
   by taking the representation of $(z_1, z_2)$ on ${\Bbb C}^r$.

 \item[(2)]
 {\it From $(m_1,m_2)$ to $({\cal F},\iota)$}$\,$:\hspace{1em}
  Given $(m_1,m_2)$ as said,
  the specification $z_1\mapsto m_1$ and $z_2\mapsto m_2$
   defines a ${\Bbb C}$-algebra homomorphism
   $\varphi_{(m_1,m_2)}:{\Bbb C}[z_1,z_2]\rightarrow \End({\Bbb C}^r)$.
  This realizes ${\Bbb C}^r$ as a ${\Bbb C}[z_1,z_2]$-module
   and, hence, an ${\cal O}_{{\Bbb A}^2}$-module ${\cal F}$,
   together with an isomorphism
    $\iota:{\Bbb C}^r\stackrel{\sim}{\rightarrow}  H^0({\cal F})$.
  Note that
   $\Supp{\cal F}=V(\Ker\varphi_{(m_1,m_2)})
     \simeq \Spec\langle{\mathbf 1}, m_1, m_2 \rangle$,
    where $\langle{\mathbf 1}, m_1, m_2 \rangle$
     is the subalgebra of $\End({\Bbb C}^r)$
     generated by ${\mathbf 1}$, $m_1$, and $m_2$.
\end{itemize}
These two operations are inverse to each other.
One should think of
 $({\cal F},\iota)$ in Item (1) and
 $\varphi_{(m_1,m_2)}$ in Item (2)
 as defining a morphism
 from the {\it fixed/rigidified} Azumaya point with the fundamental module
  $(\Spec{\Bbb C}\,,\,\End({\Bbb C}^r)\,,\,{\Bbb C}^r)$
 to (the fixed) $\Spec {\Bbb C}[z_1,z_2]={\Bbb A}^2$.
The rigidification is given by $\iota$ in Item (1) and
 by expressing the fundamental module on the Azumaya point
  explicitly as ${\Bbb C}^r$ in Item (2).

Recall from [L-Y1: Sec.~4.3] and the references quoted ibidem
the {\it commuting scheme}
 $$
  C_2M_r({\Bbb C})\;=\; \{(m_1,m_2)\,:\, m_1m_2=m_2m_1\}\;
   \subset \End({\Bbb C}^r)\times\End({\Bbb C}^r)
 $$
with the scheme structure from the standard scheme structure
 ${\Bbb A}^{r^2}$ on $\End({\Bbb C}^r)$.
This is an irreducible variety of dimension $r^2+r$.
The universal commuting pair of endomorphisms of ${\Bbb C}^r$
 is given by a section $s$ of
 ${\cal O}_{C_2M_r(\Bbb C)}
   \otimes (\End({\Bbb C}^r)\oplus\End({\Bbb C}^r))$.
The morphism
 $$
  P_{r, {\Bbb A}^2}\;:\;
   C_2M_r({\Bbb C})\; \longrightarrow\; {\mathfrak M}^{D0}_r({\Bbb A}^2)
 $$
 defined by the composition of the correspondences
 $$
  (m_1,m_2)\; \longmapsto\; s(m_1,m_2)\;
   \longmapsto\; ({\cal F},\iota)\; \longmapsto\; {\cal F}
 $$
 realizes $C_2M_r({\Bbb C})$
 as an atlas of the Artin stack ${\mathfrak M}^{D0}_r({\Bbb A}^2)$
 of D$0$-branes of type $r$ on ${\Bbb A}^2$.
The defining $\GL_r({\Bbb C})$-action on ${\Bbb C}^r$
 induces a $\GL_r({\Bbb C})$-action on $C_2M_r({\Bbb C})$
 via
 $$
  (m_1,m_2)\; \stackrel{g}{\longmapsto}\; (gm_1g^{-1}, gm_2g^{-1})\,.
 $$
An orbit of the latter action corresponds to
 an isomorphism class of morphisms
 from the {\it unfixed} Azumaya point$/{\Bbb C}$
 with a fundamental module,
 $(\pt^{\Azscriptsize},E)
  :=(\Spec{\Bbb C}, \End{E}, E)$,
  where $E\simeq {\Bbb C}^r$ abstractly,
 to (the fixed) ${\Bbb A}^2$.

\smallskip

\begin{slemma}
{\bf [closed orbit in $C^2M_r({\Bbb C})$].} {\it
 $\GL_r({\Bbb C})\cdot(m_1,m_2)$ is a closed orbit in $C_2M_r({\Bbb C})$
  if and only if $m_1$ and $m_2$ are simultaneously diagonalizable.
 The closure $\overline{\GL_r({\Bbb C})\cdot (m_1,m_2)}$ of
  every orbit $\GL_r({\Bbb C})\cdot(m_1,m_2)$ in $C_2M_r({\Bbb C})$
  contains a unique closed orbit.
}\end{slemma}

\begin{proof}
 $\GL_r({\Bbb C})\cdot(m_1,m_2)$ can be represented by $(m_1,m_2)$
  with $m_1$ and $m_2$, say, upper simultaneously triangulated.
 The $t\rightarrow 0$ limit $(m_1^0,m_2^0)$ of such $(m_1,m_2)$
  by the $1$-parameter subgroup
   $\Diag(1, t^{-1}, \,\cdots\,, t^{-(r-1)})$
  is then diagonal.
 Up to the $\GL_r({\Bbb C})$-action,
  $(m_1^0, m_2^0)$ is uniquely determined by the orbit
   $\GL_r({\Bbb C})\cdot (m_1,m_2)$
  as the former is determined by the characteristic polynomials
   $\det(\lambda I-m_1)$ and $\det(\lambda I-m_2)$
   of $m_1$ and $m_2$ up to simultaneous permutations.

\end{proof}

\smallskip

\begin{scorollary}
{\bf [categorical quotient of ${\mathfrak M}^{D0}_r({\Bbb A}^2)$].} {\it
 The categorical quotient of ${\mathfrak M}^{D0}_r({\Bbb A}^2)$,
  which by definition is the categorical quotient of
   the atlas $C_2M_r({\Bbb C})$ by the above $\GL_r({\Bbb C})$-action,
  is given by the symmetric product $S^r({\Bbb A}^2)$ of ${\Bbb A}^2$.
}\end{scorollary}

\smallskip

\begin{sremark} {\rm
A point in $S^r({\Bbb A}^2)$ corresponds to an isomorphism class
 of morphisms from $(\pt^{\Azscriptsize},E)$ to ${\Bbb A}^2$
 whose associated surrogates are reduced.
The above discussion also realizes
 $S^r({\Bbb A}^2)$
 as $\Spec (R(C_2M_r({\Bbb C}))^{GL_r({\Bbb C})})$,
 where $R(C_2M_r({\Bbb C}))^{GL_r({\Bbb C})}$
  is the ring of $\GL_r({\Bbb C})$-invariant functions
   in the function/coordinate ring $R(C_2M_r({\Bbb C}))$ of
   $C_2M_r({\Bbb C})$;
 see [Va] for a discussion in general dimensions and
  [L-Y1] for more references.
}\end{sremark}

\smallskip

Once having the categorical quotient of $C_2M_r({\Bbb C})$
  under the $\GL_r({\Bbb C})$-action,
 and hence of the stack ${\mathfrak M}^{D0}_r({\Bbb A}^2)$,
it is natural to attempt to follow [Ki] and [M-F-K],
  and the related discussion in [Na]
 to consider a geometric-invariant-theory setting
 on the $\GL_r({\Bbb C})$-variety $C_2M_r({\Bbb C})$
 to produce birational models of $S^r({\Bbb A}^2)$.
The would-be procedure goes as follows:
Let $L$ be the trivial line bundle ${\cal O}_{C_2M_r({\Bbb C})}$
  on $C_2M_r({\Bbb C})$
 with a linearization of the $\GL_r({\Bbb C})$-action
 through
  a character $\chi:\GL_r({\Bbb C})\rightarrow {\Bbb C}^{\times}$
  via the determinant function $\det$
  to some positive power:
 $g\cdot((m_1,m_2), z)=(g\cdot (m_1,m_2), \chi(g)^{-1}z)$.
A would-be birational model of $S^r({\Bbb A}^2)$
 is then obtained by taking the categorial quotient
 $(C_2M_r({\Bbb C}))^{ss}(\chi)/\!\!\sim$ of
 the $\GL_r({\Bbb C})$-action on the semistable locus
 $C_2M_r({\Bbb C})^{ss}(\chi)$ in $C_2M_r({\Bbb C})$
 with respect to the linearization on $L$ specified by $\chi$.
However, this won't work here
 as $C_2M_r({\Bbb C})^{ss}(\chi)= \emptyset$.
The reason is that
 the stabilizer of the $\GL_r({\Bbb C})$-action on $C_2M_r({\Bbb C})$
  at a point on a principal orbit is isomorphic to
  $({\Bbb C}^{\times})^r$;
this is too big to allow the $\chi$-linearized $L$ to have
 $\GL_r({\Bbb C})$-invariant sections except the zero-section.
This is why we need to employ
 a netted categorical quotient construction,
 instead of a direct GIT-quotient construction.
We will use a construction of Nakajima in [Na]
 to guide our netted categorical quotient construction
 relevant to our goal.

To remedy the above issue,
consider instead
 the total space ${\Bbb E}$
  of the universal fundamental module
  ${\cal O}_{C_2M_r({\Bbb C})}\otimes{\Bbb C}^r$
  on $C_2M_r({\Bbb C})$
 and employ the above GIT setting to ${\Bbb E}\,$:
\begin{itemize}
 \item[$\bullet$]
  {\it the $\GL_r({\Bbb C})$-action on ${\Bbb E}\,$}:\hspace{.6em}
  The $\GL_r({\Bbb C})$-action on $C_2M_r({\Bbb C})$
   lifts to a $\GL_r({\Bbb C})$-action on ${\Bbb E}$
   $$
    (m_1,m_2; v)\;
     \stackrel{g}{\longmapsto}\; (gm_1g^{-1}, gm_2g^{-1}; gv)\,.
   $$
  The stabilizer of $(m_1,m_2;v)\in $ a principal orbit in ${\Bbb E}$
   is now trivial.

 \item[$\bullet$]
  {\it the line bundle $L$ and its linearization}$\,$:\hspace{1em}
  Let $L$ be the trivial line bundle ${\cal O}_{\Bbb E}$
   with the linearlization specified by a character $\chi$
   as in the previous discussion:
   $$
    ((m_1,m_2; v)\,,\, z)\;
     \stackrel{g}{\longmapsto}\;
     (g\cdot (m_1,m_2; v)\,,\, \chi(g)^{-1}z)\,.
   $$
\end{itemize}
Thus, the above GIT-construction setting of [Ki] and [M-F-K]
  is now more likely to be applicable
  to the $\GL_r({\Bbb C})$-variety ${\Bbb E}$
 and, indeed, the detail is worked out by Nakajima in [Na].

\smallskip

\begin{stheorem}
{\bf [Hilbert scheme $({\Bbb A}^2)^{[r]}$ from GIT-quotient].}
{\rm ([Na: Theorem 1.9 and Lemma 3.25].)}
{\it
 The GIT-quotient
   ${\Bbb E}\,/\!/_{\chi}\GL_r({\Bbb C})
    := {\Bbb E}^{ss}(\chi)/\GL_r({\Bbb C})$
   of ${\Bbb E}$
  with respect to $L$ with the linearization specified by $\chi$
  is canonically isomorphic to the Hilbert scheme $({\Bbb A}^2)^{[r]}$
  of $0$-dimensional subschemes of length $r$ on ${\Bbb A}^2$.
}\end{stheorem}

\smallskip

\begin{slemma}
{\bf [netted categorical quotient on $C_2M_r({\Bbb C})$].}
{\rm ([L-Y1: Proposition 4.3.3].)}
 Let $\pi:{\Bbb E}\rightarrow C_2M_r({\Bbb C})$
  be the defining projection morphism.
 Then
  $$
   \pi({\Bbb E}^{ss}(\chi))/\GL_r({\Bbb C})\;
    \simeq\; {\Bbb E}^{ss}(\chi)/\GL_r({\Bbb C})\;
    \simeq\; ({\Bbb A}^2)^{[r]}
  $$
 canonically.
\end{slemma}

\smallskip

\noindent
This is a rephrasing of [L-Y1: Proposition 4.3.3];
 we give another proof below that fits better the current setting.
Note that,
 as $C_2M_r({\Bbb C})$ is irreducible and
  $\pi:{\Bbb E}\rightarrow C_2M_r({\Bbb C})$ is an open morphism,
 $\pi({\Bbb E}^{ss}(\chi))$ is open and dense in $C_2M_r({\Bbb C})$.

\smallskip

\begin{proof}
 Let
  $R({\Bbb E})$ be the coordinate ring of the affine variety ${\Bbb E}$
   and
  $R({\Bbb E})^{\chi,n}$ be the ${\Bbb C}$-subspace of $R({\Bbb E})$
   that consists of $f\in R({\Bbb E})$ that satisfies
    $f(g\cdot(m_1,m_2; v))= \chi(g)^n f(m_1,m_2; v)$.
 Then, by definition,
  $(m_1,m_2; v)\in {\Bbb E}^{ss}(\chi)$
  if there exists an $f\in R({\Bbb E})^{\chi,n}$ with $n\ge 1$
   such that $f(m_1, m_2; v)\ne 0$.
 In our case,
   as $\chi$ is a positive power of the determinant function on
    $\GL_r({\Bbb C})$ and
   the stabilizer $\Stab(\,\cdot\,)$ is an open subset
    of an affine space,
  such an $f$ can exist only if $\Stab(m_1, m_2; v)$ is trivial.
 Thus ${\Bbb E}^{ss}(\chi)$
  is identical to the stable locus ${\Bbb E}^s(\chi)$
  of $\GL_r({\Bbb C})$-action on ${\Bbb E}$ and
  the quotient
   ${\Bbb E}^{ss}(\chi) \rightarrow {\Bbb E}^{ss}(\chi)/\GL_r({\Bbb C})$
   is a geometric quotient;
  indeed, a $\GL_r({\Bbb C})$-bundle.

 For convenience,
  express ${\Bbb E}$ canonically as $C_2M_r({\Bbb C})\times{\Bbb C}^r$
   in the analytic language/notation.
 Then, for $(m_1, m_2, v)\in {\Bbb E}^{ss}(\chi)$,
  the projection map
   $\GL_r({\Bbb C})\cdot(m_1, m_2; v)\rightarrow {\Bbb C}^r$
   is a bundle map over ${\Bbb C}^r-\{{\mathbf 0}\}$
   with fiber the inhomogeneous general linear group
   $\IGL_{r-1}({\Bbb C})$
   (i.e.\ the affine transformation group
           of the vector space ${\Bbb C}^{r-1}$).
 This shows
  that for any $(m_1, m_2; v)\in {\Bbb E}^{ss}(\chi)$,
  $$
   \dimm
    ( \GL_r({\Bbb C})\cdot (m_1, m_2; v)
      \cap (C_2M_r({\Bbb C})\times\{v\}) )\;
   =\;  r^2-r
  $$
  and that, for any fixed $v_0\ne {\mathbf 0}$,
  $$
   {\Bbb E}^{ss}(\chi)/\GL_r({\Bbb C})\;
    =\; ({\Bbb E}^{ss}(\chi)\cap (C_2M_r({\Bbb C})\times\{v_0\}))
          /\IGL_{r-1}({\Bbb C})\,.
  $$

 Identify $\Stab(v_0)$ of the $\GL_r({\Bbb C})$-action on ${\Bbb C}^r$
  with $\IGL_{r-1}({\Bbb C})$.
 As, for $(m_1, m_2, v_0)\in {\Bbb E}^{ss}(\chi)$,
  \begin{eqnarray*}
   \lefteqn{
    \IGL_{r-1}({\Bbb C})\cdot(m_1, m_2; v_0)\;
     =\; \GL_r({\Bbb C})\cdot(m_1, m_2; v_0)
            \cap (C_2M_r({\Bbb C})\times\{v_0\})   }\\
   && \simeq\;
       \pi(\IGL_{r-1}({\Bbb C})\cdot(m_1, m_2; v_0))  \\
   && \subset\;  \pi(\GL_r({\Bbb C})\cdot(m_1, m_2; v_0))\;
       \subset\;  \GL_r({\Bbb C})\cdot (m_1, m_2)\,,
  \end{eqnarray*}
  and $\IGL_{r-1}({\Bbb C})\cdot(m_1, m_2; v_0)$ has dimension $r^2-r$,
 the orbit $\GL_r({\Bbb C})\cdot (m_1, m_2)$ in $C_2M_r({\Bbb C})$
  must also have dimension $r^2-r$
  and the dimension of $\Stab(m_1, m_2)$ must be $r$.
 As $\Stab(m_1, m_2; v_0)$ is trivial,
  this implies that
   $\Stab(m_1, m_2)\cdot v_0$ is open dense in ${\Bbb C}^r$.
 Translating this to any
  $(m_1^{\prime}, m_2^{\prime}; v_0)
   \in \IGL_{r-1}({\Bbb C})\cdot (m_1, m_2; v_0)$
  via the $\IGL_{r-1}({\Bbb C})$-action,
 this implies that
  $$
   \GL_r({\Bbb C})\cdot (m_1, m_2; v_0)\;
   \subset\;
    \pi^{-1}( \pi( \IGL_{r-1}({\Bbb C})\cdot(m_1, m_2; v_0) ) )\;
   =\; \pi(\IGL_{r-1}\cdot (m_1, m_2; v_0))\times{\Bbb C}^r\,.
  $$
 Since
  $\pi(\GL_r({\Bbb C})\cdot(m_1, m_2; v_0))
   =\GL_r({\Bbb C})\cdot (m_1, m_2)$,
 one concludes that
  $$
   \pi( \GL_r({\Bbb C})\cdot (m_1, m_2; v_0)
         \cap (C_2M_r({\Bbb C})\times\{v_0\}) )\;
   =\; \GL_r({\Bbb C})\cdot (m_1, m_2)\,.
  $$
 This implies that
  the isomorphism
   $\pi: {\Bbb E}^{ss}(\chi)\cap (C_2M_r({\Bbb C})\times\{v_0\})
      \stackrel{\sim}{\rightarrow} \pi({\Bbb E}^{ss}(\chi))$
   is equivariant under the group homomorphism
   $\IGL_{r-1}({\Bbb C})(=\Stab(v_0)) \hookrightarrow \GL_r({\Bbb C})$
  with isomorphic orbit-spaces.
 The lemma follows.

\end{proof}

\smallskip

\begin{sremark}
{\rm [{\it characterization of $\pi({\Bbb E}^{ss}(\chi))$}].} {\rm
 The image $\pi({\Bbb E}^{ss}(\chi))$ in $C_2M_r({\Bbb C})$
  consists of $(m_1, m_2)$
  such that there exists a $v\in {\Bbb C}^r$
   with ${\Bbb C}[m_1, m_2]\cdot v = {\Bbb C}^r$.
 This happens if and only if
  the subalgebra $\langle {\mathbf 1}, m_1, m_2\rangle$
  of $\End({\Bbb C^r})$
   is isomorphic to ${\Bbb C}^r$ as ${\Bbb C}$-vector spaces
   (and, hence,
    as $\langle {\mathbf 1}, m_1, m_2\rangle$-modules as well).
}\end{sremark}

\smallskip

\begin{sremark}
 The induced morphism
  $({\Bbb A}^2)^{[r]}\rightarrow S^r({\Bbb A}^2)$
  through the construction is the Hilbert-Chow morphism.
\end{sremark}

\bigskip

\begin{flushleft}
{\bf D0-branes on $[{\Bbb A}^2/\Gamma]$  and
     an atlas for ${\mathfrak M}^{D0}_1([{\Bbb A}^2/\Gamma])$.}
\end{flushleft}
The $\Gamma$-action on ${\Bbb C}[z_1, z_2]$ induces
 a $\Gamma$-action on $C_2M_r({\Bbb C})$ via
 $$
  (m_1,m_2) \;
  \stackrel{\scriptsize
            \mbox{\raisebox{3ex}{
              $\gamma\,
               =\,\left(\begin{array}{cc} a&c\\ b&d \end{array} \right)$
             }}} {\longlongmapsto}\;
  (m_1,m_2)
  \left(\begin{array}{cc} a & c \\ b & d\end{array} \right)^t\;
   =:\; \gamma\odot (m_1, m_2).
 $$
In terms of analytic expression,
 regard $C_2M_r({\Bbb C})$ as a subset in the vector space
 $\End({\Bbb C}^r)\otimes {\Bbb C}^2$;
then,
 the $\Gamma$-action on ${\Bbb C}^2$ induces a $\Gamma$-action
 on $\End({\Bbb C}^r)\otimes {\Bbb C}^2$
 that leaves $C_2M_r({\Bbb C})$ invariant.
This gives the above action.
Here, ${\Bbb C}^2$ is identified with the vector space
 $\Span\{z_1, z_2\}$,
 with ${\Bbb A}^2=\Spec\Sym^{\bullet}({\Bbb C}^2)$.
Note that the $\Gamma$- and the $\GL_r({\Bbb C})$-action
 on $C_2M_r({\Bbb C})$ commute:
 $$
  \gamma\odot((g\cdot(m_1, m_2)))\;
   =\; g\cdot(\gamma\odot (m_1, m_2))
 $$
for $(m_1, m_2)\in C_2M_r({\Bbb C})$, $g\in \GL_r({\Bbb C})$,
 and $\gamma\in\Gamma$.
In terms of this action,
a rigidified $\Gamma$-${\cal O}_{{\Bbb A}^2}$-module
 $({\cal F},\iota)$ on the $\Gamma$-variety ${\Bbb A}^2$,
  where $\iota:{\Bbb C}^r\stackrel{\sim}{\rightarrow} H^0({\cal F})$,
 is given by a triple $(m_1, m_2; \rho)$,
 where $\rho:\Gamma^{\circ}\rightarrow \GL_r({\Bbb C})$
  is a representation of $\Gamma^{\circ}$,
  that satisfies the ${\cal O}_{{\Bbb A}^2}$-linearity condition
   of the $\Gamma$-action on ${\cal F}$,
  now expressed as
 $$
  \gamma\odot(m_1, m_2)\;=\; \rho(\gamma)\cdot(m_1, m_2)
   \hspace{1em}\mbox{for all $\gamma\in \Gamma$}\,.
 $$
Here $\Gamma^{\circ}$ is the dual group of $\Gamma$,
 which has the same elements as in $\Gamma$
  but with $\gamma_1\gamma_2$ in $\Gamma^{\circ}$ defined to be
   $\gamma_2\gamma_1$ in $\Gamma$.
The gluing condition of $({\cal F}.\iota)$ gives
 a correspondence $\rho:\Gamma^{\circ}\rightarrow \GL_r({\Bbb C})$
while the cocycle condition on the gluing imposes further
 that $\rho$ is a group-homomorphism.

Let
 $\Rep_{\Gamma^{\circ}}({\Bbb C}^r)\subset \prod_r\GL_r({\Bbb C})$
  be the representation variety of $\Gamma^{\circ}$ into $\GL_r({\Bbb C})$.
This is a $\GL_r({\Bbb C})$-scheme under
  $(\Ad_g\cdot\rho)(\gamma)=g\rho(\gamma) g^{-1}$ for $\gamma\in\Gamma$,
 where
  $\Ad_{\bullet}$ is the adjoint action of $\GL_r({\Bbb C})$ on itself.
The McKay correspondence gives
 a bijection between the set of equivalence classes of
  irreducible representations of $\Gamma$
  and the set of vertices of the extended Dynkin diagram
  associated to $\Gamma$.
Any other representation of $\Gamma$ is a direct sum of these
 irreducible representations.
It follows that $\Rep_{\Gamma^{\circ}}({\Bbb C}^r)$
 consists of finitely many $\GL_r({\Bbb C})$-orbits.
Let
 \begin{eqnarray*}
  C_2\Gamma M_r({\Bbb C})
   & :=
   & \{(m_1,m_2;\rho)\;:\;
      \gamma\odot(m_1, m_2)\;=\; \rho(\gamma)\cdot(m_1, m_2)
       \hspace{.6em}\mbox{for all $\gamma\in \Gamma$}\, \}\\
  & \subset
   & C_2M_r({\Bbb C})\times \Rep_{\Gamma^{\circ}}({\Bbb C}^r)
 \end{eqnarray*}
 with the subscheme structure the same as the subscheme structure
  defined by these algebraic constraint equations
  on $C_2M_r({\Bbb C})\times \Rep_{\Gamma^{\circ}}({\Bbb C}^r)$.
The tautological
 $\Gamma$-${\cal O}_{C_2\Gamma M_r({\Bbb C})\times{\Bbb A}^2}$-module
 on $C_2\Gamma M_r({\Bbb C})\times{\Bbb A}^2$
 defines a unique morphism
 $$
  P_{1,[{\Bbb A}^2/\Gamma]}\;:\;
   C_2\Gamma M_r({\Bbb C})\;
     \longrightarrow\; {\mathfrak M}^{D0}_1([{\Bbb A}^2/\Gamma])\,.
 $$
This gives an atlas of the Artin stack
 ${\mathfrak M}^{D0}_1([{\Bbb A}^2/\Gamma])$.
Change of rigidifications of $\Gamma$-${\cal O}_{{\Bbb A}^2}$-modules
  on ${\Bbb A}^2$
 induces a $\GL_r({\Bbb C})$-action on $C_2\Gamma M_r({\Bbb C})$
  by
  $$
   (m_1,m_2; \rho)\;
    \stackrel{g}{\longmapsto}\; (gm_1g^{-1}, gm_2g^{-1}; \Ad_g\cdot\rho)\,.
  $$
There is a bijection between
 the set of $\GL_r({\Bbb C})$-orbits in $C_2\Gamma M_r({\Bbb C})$  and
 the set of isomorphism classes of
  $\Gamma$-${\cal O}_{{\Bbb A}^2}$-modules of length $r$
  on the $\Gamma$-variety ${\Bbb A}^2$.

\bigskip

\begin{flushleft}
{\bf Resolution of the ADE orbifold singularity of
     ${\Bbb A}^2/\Gamma$
     via ${\mathfrak M}^{D0}_1([{\Bbb A}^2/\Gamma])$.}
\end{flushleft}
We now proceed to extract a $\GL_r({\Bbb C})$-invariant locus
 from the atlas $C_2\Gamma M_r({\Bbb C})$
 of ${\mathfrak M}^{D0}_1([{\Bbb A}^2/\Gamma])$
 whose geometric quotient gives the minimal smooth resolution
 of ${\Bbb A}^2/\Gamma$.

Recall the $\Gamma$-action on $C_2M_r({\Bbb C})$.
It leaves $\pi({\Bbb E}^{ss}(\chi))$ invariant
 and descends to the natural $\Gamma$-action on
 $\pi({\Bbb E}^{ss}(\chi))/\GL_r({\Bbb C})$
 via the canonical isomorphism with $({\Bbb A}^2)^{[r]}$.

\smallskip

\begin{stheorem}
{\bf [$\Gamma$-invariant subscheme].}
{\rm (Ginzburg-Kapranov, Ito-Nakamura,
      [Na: Theorem~4.1 and Theorem 4.4].)}
 Let $(({\Bbb A}^2)^{[r]})^{\Gamma}$ be the fixed-point locus
  of the $\Gamma$-action on $({\Bbb A}^2)^{[r]}$.
 Then
  $$
   (({\Bbb A}^2)^{[r]})^{\Gamma}\;
       =\;  W\,\amalg\,(\mbox{a finite set of points})\,,
  $$
  where $W$ is the minimal smooth resolution of
   ${\Bbb A}^2/\Gamma\simeq (S^r({\Bbb A}^2))^{\Gamma}$
   via the Hilbert-Chow morphism.
\end{stheorem}

\smallskip

Let $\widetilde{Z}\subset W\times {\Bbb A}^2$
 be the universal subscheme on $(W\times{\Bbb A}^2)/W$
 associated to $W$.
Then
 ${\cal O}_{\widetilde{Z}}$ is
  a $W$-family of ${\Gamma}$-${\cal O}_{{\Bbb A}^2}$-modules
  of length $r$ on ${\Bbb A}^2$ and, hence,
 defines a morphism
  $p^W: W\rightarrow {\mathfrak M}^{D0}_1([{\Bbb A}^2/\Gamma])$.
Consider the fibered product
 \begin{eqnarray*}
  \xymatrix{
   C_2\Gamma M_r({\Bbb C})
     \times_{P_{1,[{\Bbb A}^2/\Gamma]}\,,\,
            {\mathfrak M}^{D0}_1([{\Bbb A}^2/\Gamma])\,,\, p^W}
    W\ar[d]_{pr_1}\ar[rrrr]^{pr_2}
    &&&& W\ar[d]^{p^W}\\
  C_2\Gamma M_r({\Bbb C})\ar[rrrr]^{P_{1,[{\Bbb A}^2/\Gamma]}}
    &&&& {\mathfrak M}^{D0}_1([{\Bbb A}^2/\Gamma])
  }
 \end{eqnarray*}
 and
let
 $$
  C_2\Gamma M_r({\Bbb C})^{\circ} \;
  :=\; \pr_1\left(
        C_2\Gamma M_r({\Bbb C})
         \times_{P_{1,[{\Bbb A}^2/\Gamma]}\,,\,
                {\mathfrak M}^{D0}_1([{\Bbb A}^2/\Gamma])\,,\, p^W}
      W \right)\;
   \subset\; C_2\Gamma M_r({\Bbb C})\,.
 $$
Then, it follows from
 Theorem~2.9
       % Theorem [$\Gamma$-invariant subscheme]
 and the surjectivity of $\pr_2$ in the above fibered product diagram
 that:

\smallskip

\begin{scorollary}
{\bf [D-brane probe resolution of orbifold singulary].}
 The categorical quotient
  $C_2\Gamma M_r({\Bbb C})^{\circ}/\GL_r({\Bbb C})$
  is a geometric quotient and
 is isomorphic to $W$.
 Thus, the moduli stack ${\mathfrak M}^{D0}_1([{\Bbb A}^2/\Gamma])$ of
  D0-branes on the orbifold $[{\Bbb A}^2/\Gamma]$ of stacky type $1$
  \`{a} la Polchinski-Grothendieck Ansatz contains
  a substack
   whose associated coarse moduli space is isomorphic to
   the minimal resolution $\widetilde{{\Bbb A}^2/\Gamma}$
   of the variety ${\Bbb A}^2/\Gamma$
   with an ADE orbifold singularity.
\end{scorollary}

\smallskip

\begin{sremark}
{\rm [{\it characterization of $C_2\Gamma M_r({\Bbb C})^{\circ}$}].}
{\rm
 It follows from Lemma~2.6
               % Lemma [netted categorical quotient
               %        on $C_2M_r({\Bbb C})$]
  that the forgetful morphism
   $C_2\Gamma M_r({\Bbb C})\rightarrow C_2M_r({\Bbb C})$
   sends $C_2\Gamma M_r({\Bbb C})^{\circ}$
   onto $\pi({\Bbb E}^{ss}(\chi))$.
 It follows from [Na: Theorem~4.4] and
   Remark~2.7
        % Remark [characterization of $\pi({\Bbb E}^{ss}(\chi))$]
  that a geometric point in
   $C_2\Gamma M_r({\Bbb C})^{\circ}/\GL_r({\Bbb C})$
   corresponds to an isomorphism class of morphisms
   $$
    \varphi\;:\;
    (\pt^{A\!z}, {\Bbb C}^r)\,
    :=\, ((\Spec{\Bbb C}, \End({\Bbb C}^r)),  {\Bbb C}^r)\;
     \longrightarrow\; [{\Bbb A}^2/\Gamma]
   $$
   from the Azumaya point with the fundamental module
   to the orbifold that satisfy:
    \begin{itemize}
     \item[$\cdot$]
      the induced $\Gamma$-action/representation on the Chan-Paton module
       ${\Bbb C}^r$ is the regular representation;

     \item[$\cdot$]
      there exists a $\Gamma$-fixed element $v_0\in {\Bbb C}^r$
       such that
        $H^0({\cal O}_{pt_{\varphi}})\cdot v_0={\Bbb C}^r$,
       where $\pt_{\varphi}$, a $0$-dimensional scheme of length $r$,
        is the surrogate of $\pt^{A\!z}$ associated to $\varphi$.
    \end{itemize}
   And vice versa.
 In particular, any morphism
  $\varphi:(\pt^{A\!z},{\Bbb C}^r)\rightarrow [{\Bbb A}^2/\Gamma]$
  of stacky type $1$ whose image is not the orbifold-point
   with structure group $\Gamma$
  satisfies the above conditions and
 its isomorphism class is contained in
  $C_2\Gamma M_r({\Bbb C})^{\circ}/\GL_r({\Bbb C})$.
}\end{sremark}

\smallskip

\begin{sremark}
{\rm [{\it movable vs.\ unmovable morphism/D-branes}].}
{\rm
 Morphisms associated to geometric points in
  $C_2\Gamma M_r({\Bbb C})^{\circ}$ are movable
  in the sense that they can be deformed to have
  the image D-brane anywhere on $[{\Bbb A}^2/\Gamma]$.
 However, in general, there are morphisms from
  D0-branes to $[{\Bbb A}^2/\Gamma]$
  that are trapped at the orbifold-point
   that corresponds to the singular point of ${\Bbb A}^2/\Gamma$.
 Here, ``trapped" means that these morphisms cannot be deformed
  to make the image D-brane move away from this orbifold point.
 It is unknown to us
  whether such trapped/unmovable D-branes at the singularity
   have any stringy significance/implication.
}\end{sremark}

\newpage
\baselineskip 13pt
%references
%endfootnotesize

\end{document}